\documentclass[12pt,a4paper]{article}%
\usepackage{amsmath}
\usepackage{graphicx}
\usepackage{epsfig}%
\usepackage{amsfonts}%
\usepackage{amssymb}

\usepackage{color}
\usepackage[normalem]{ulem}

\newtheorem{theorem}{Theorem}[section]

\newtheorem{definition}[theorem]{Definition}

\newtheorem{lemma}[theorem]{Lemma}

\newtheorem{remark}[theorem]{Remark}

\newenvironment{proof}[1][Proof]{\noindent \emph{#1.} }{\hfill \ 
\rule{0.5em}{0.5em}}

\makeatletter\@addtoreset{equation}{section}\makeatother
\makeatletter\@addtoreset{figure}{section}\makeatother
\makeatletter\@addtoreset{table}{section}\makeatother
\begin{document}

\title{Range-separated  tensor   representation 
of  the discretized multidimensional Dirac delta and \\ elliptic operator inverse}

\author{
        Boris N. Khoromskij \thanks{Max Planck Institute for
        Mathematics in the Sciences, Inselstr.~22-26, D-04103 Leipzig,
        Germany ({\tt bokh@mis.mpg.de}).}\\
 }

  \date{}

\maketitle

\begin{abstract}
In this paper,  we introduce the  operator dependent
range-separated tensor approximation of the discretized Dirac delta in $\mathbb{R}^d$.
It is constructed by application of the discrete elliptic operator to the 
range-separated decomposition of the associated Green kernel discretized on the 
Cartesian grid in $\mathbb{R}^d$.
The presented operator dependent local-global splitting of the Dirac delta can be applied 
for solving the potential equations in non-homogeneous medium 
when the density in the right-hand side is given by the large sum of pointwise singular charges.
We show how the idea of the operator dependent RS splitting of the Dirac delta can be extended to the closely 
  related problem on the range separated tensor representation of the elliptic resolvent.
The numerical tests confirm the expected localization properties of the 
obtained operator dependent approximation of the Dirac delta represented on a tensor grid.
As an example of application, we consider the regularization scheme  for solving the Poisson-Boltzmann 
equation for modeling the electrostatics in bio-molecules.
\end{abstract}



\section{Introduction} \label{sec:Intro}

The grid-based approximation of the  multidimensional Dirac delta function 
arises in modeling of the long-range potentials in multiparticle systems in variety of applications.
 For example, in potential equations describing electrostatics of large biomolecules,
 in molecular dynamics simulations of large solvated biological systems,  
 in docking or folding of proteins, pattern recognition  and 
 many other problems  \cite{HuMcCam:1999,LiStCaMaMe:13,Maday:2018}.
  Furthermore, in PDE driven models via elliptic operators in $\mathbb{R}^d$ with 
  constant coefficients the Dirac delta naturally arises as 
 the singular point density specifying the definition of the corresponding 
 fundamental solution (the Green kernel). 
 %In this way, we rely on the concept of operator dependent Dirac delta.  
  In what follows, we rely on the new concept of \emph{operator dependent Dirac delta}
 in elliptic problems, where this function can be associated, for example, with 
 the harmonic or bi-harmonic operator, in representation of elastic and hydrodynamics 
 potentials in 3D elasticity problems, or in the fundamental solution 
 of advection-diffusion operator with constant coefficients in $\mathbb{R}^3$.

  In the approaches based on the potential energy ansatz, the potential elliptic 
  equations usually include 
  the ``multiparticle'' Dirac delta function,   which is   used for modeling the density of charges,
  smeared over the compact part of the computational domain in the course of the solution process.
 Due to strong singularity in the Dirac delta such computational schemes are error prone owing to 
 strict limitations on the 3D grid size, when using the traditional grid based numerical methods.
 These restrictions lead to quite heavy computation task using rather coarse available grids,
 considerably reducing the capabilities of numerical simulations when using 
 the Dirac delta function in modeling of many-particle systems. The attempts based on analytic 
 pre-computation of the free space component in the collective potential on a 3D grid may cause   
 some other numerical troubles.

 In this paper, we introduce the new method for the operator dependent approximation of the 
 Dirac delta in $\mathbb{R}^d$ by using the range-separated (RS) tensor representation 
 \cite{BeKhKh_RS:16,BeKhKh_RS_SISC:18}
 of the Green kernel for certain elliptic operator discretized   %, say for $d=3$,  
  over tensor product Cartesian grids with the large grid size.  
  To fix the idea, we consider the radial function\footnote{If there is no confusion, 
we denote the  multivariate radial function $p(x)$, $x \in \mathbb{R}^d$ by the same 
symbol $p(\|x\|)$, when considered as the univariate function of $\|x\|$.}
$p(x)=p(\|x\|)$ specifying the Green's kernel (fundamental solution) for  a second order 
elliptic operator $\cal L$ with constant coefficients, such that the equation 
\begin{equation}\label{eqn:OD_Dirac}
 {\cal L} p(\|x\|) = \delta(\|x\|), \quad  x=(x_1,\ldots,x_d) \in \mathbb{R}^d
\end{equation}
holds, where $\delta(x)$ is the operator dependent Dirac delta function 
in $\mathbb{R}^d$ (tempered distribution).
Assume that both the operator 
${\cal L}$ and the Green kernel $p(\|x\|)$ are represented in a low-rank tensor 
formats on the $d$-fold $n \times  \cdots \times n$ Cartesian grid in $\mathbb{R}^d$.
Now for given RS type tensor decomposition of the discretized Green kernel 
\begin{equation}\label{eqn:RS_Green}
p(\|x\|)=p_s(\|x\|) + p_l(\|x\|),
\end{equation} 
we introduce the RS tensor splitting of the discretized Dirac delta associated with the 
operator ${\cal L}$,
\begin{equation}\label{eqn:RS_delta}
 \delta(\|x\|)={\cal L} p_s(\|x\|) + {\cal L} p_l(\|x\|) := \delta_s(\|x\|)+ \delta_l(\|x\|),
\end{equation}
where both terms in the right-hand side can be rewritten in the low rank tensor form.
Here $\delta_s(\|x\|)$ and $\delta_l(\|x\|)$ represent the grid-based decomposition of 
the Dirac delta  into the  \emph{short- and long-range components}, where the latter is smooth enough.
    
    This  introduces the smoothing operation for the discretized 
    Dirac delta, such that the target function $\delta(\|x\|)$ is recovered from the smoothed version
    by the highly localized short-range correction. 
   The total error in the tensor decomposition is controlled 
   by the $\varepsilon$-truncation rank parameters and by the  
   one-dimensional size $n$ of the  representation grid in $\mathbb{R}^d$. 
    Using tensor approach, the grid-based  multidimensional Dirac delta is presented 
    with a controllable accuracy level, which is practically 
    not possible in the traditional numerical methods.
    Notice that the main idea of the approach was reported  by the author
    in a brief form in the monograph \cite{KhorBook:17}, \S5.8.1.
  
  The idea of the operator dependent RS splitting of the Dirac delta can be extended to the closely 
  related problem on the range separated tensor representation of the elliptic resolvent.
  In Section \ref{sec:RS_Resolvent}, we discuss the concept of the RS splitting 
for the elliptic operator inverse and present some examples demonstrating the application of 
this new approach to variety of elliptic operators with constant coefficients.
   
  Our techniques can be extended to the RS splitting of the multiparticle Dirac delta.
In this case, the calculation of the long-range part in the collective Dirac delta can be viewed as 
the smoothing procedure for highly singular functions in the right-hand site of the potential equation.
Specifically, for a given  generating radial function
$p(x)=p(\|x\|)$, $x \in \mathbb{R}^3$ the calculation and meshing up of a weighted sum 
of  interaction potentials in the  large $N$-particle system,
with the particle locations $x_\nu \in \mathbb{R}^3$, and charges $q_\nu\in \mathbb{R}$, $\nu=1,...,N$, 
\begin{equation}\label{eqn:PotSum1}
 P_0(x)= {\sum}_{\nu=1}^{N} {q_\nu}\,p({\|x-x_\nu\|}), %\quad q_\nu \in \mathbb{R},
  \quad x_\nu, x \in \Omega=[-b,b]^3,
\end{equation}
leads to computationally intensive numerical task. 
Indeed, the commonly used generating radial functions $p(\|x\|)$ exhibit  a slow polynomial 
decay $O(\|x\|^{-\alpha})$, $\alpha>0$  as $\|x\| \to \infty$  so that  each individual term 
in (\ref{eqn:PotSum1}) contributes essentially to the total potential 
at each point in the computational box $\Omega$.  This anticipates the 
$O(N)$ complexity for the straightforward summation at every fixed target $x \in \mathbb{R}^3$,
resulting in $O(N^2)$ complexity for the energy calculation and $O(n^3 N)$ cost for the meshing up 
the $n\times n \times n$ grid.
Moreover, in many cases, 
the radial function $p(\|x\|)$ has a singularity or a cusp at the origin, $x=0$, making 
its accurate grid representation problematic. 

 Our method allows calculation of the long-range part of the ``collective `` Dirac delta  
    in the RS tensor format, which is implemented via one-dimensional operations and with 
    the linear complexity in the number of particles $N$ requiring only $O(n \log^q N)$ storage size.   
      
The commonly used example of the harmonic Green kernel is given by 
\[
 p(\|x\|) = \frac{\Gamma(d/2-1)}{4\pi^{d/2} \|x\|^{d-2}}, 
 \quad \mbox{for} \quad {\cal L}=-\Delta,\quad x\in \mathbb{R}^d,
\]
%for ${\cal L}=-\Delta$, 
where $\Delta$ denotes the Laplace operator in $\mathbb{R}^d$
\[
 \Delta = \sum\limits_{\ell=1}^d \frac{\partial^2}{\partial x^2_\ell}.
\]
For $d=3$ the function $p(\|x\|)$ coincides with the classical Newton kernel 
\begin{equation}\label{eq:Newton}
p(\|x\|) = \frac{1}{4\pi \|x\|} \quad \mbox{(the Coulomb potential).}
\end{equation}
In this  case the solution $u(x)$ of the Poisson equation (see also (\ref{eqn:L-potential}))
\[
 -\Delta u = f
\]
is called the \emph{harmonic potential} of the density $f$. 

The Yukawa and Helmholtz kernels in $\mathbb{R}^3$ are given by 
$\frac{e^{-\kappa \|x\|}}{4\pi\|x\|}$ corresponding to the choice ${\cal L}=-\Delta \pm \kappa^2$, 
$\kappa\in \mathbb{C}$. The large class of the so-called  Mat\`ern radial functions 
(see \cite{LKKh2M:17} and references therein) includes the Slater
function $e^{-\kappa \|x\|}$, $\kappa \in (0,\infty)$ and the dipole-dipole interaction kernel
$1/\|x\|^3$, for $x \in \mathbb{R}^3$.

 Tensor-based approach provides efficient computation of the 
collective long-range potential of a many-particle system in homogeneous media, 
see for example assembled tensor method
\cite{VeBokh_CPC:14,VeBokh_NLAA:16} for finite lattices, or the RS tensor format 
\cite{BeKhKh_RS:16,BeKhKh_RS_SISC:18}
for the many-particle systems of general type. The collective electrostatic potential 
on a 3D lattice is computed  by one-dimensional translations and summations, while the RS tensor 
method provides computations with linear complexity in the number of particles, $O(N)$ and 
linear-logarithmic cost the univariate grid size, $O(n \log ^q n)$, but not as $O(n^3)$.
However, development of these methods for biomolecular modeling is a forthcoming problem,
and the present paper makes a step forward to the efficient solution of this problem.  

The modern techniques for computation of the collective electrostatic potential in heterogeneous media 
in biomolecular modeling  (as well as in docking and classification problems) 
  are based on numerical solution of the elliptic equations  in $\mathbb{R}^3$
with non-constant permeability, where the right-hand side is represented 
by a large sum of weighted delta functions (see example of the computational domain
in Figure \ref{fig:Protein}). In this model the solute (molecular) region $\Omega_m$ 
consists of the  union of small balls of radius $\sigma > 0$ around  the atomic centers
(see \S\ref{ssec:PBE_Applic1} for the detailed discussion).

 \begin{figure}[htb]
\centering
\includegraphics[width=5.5cm]{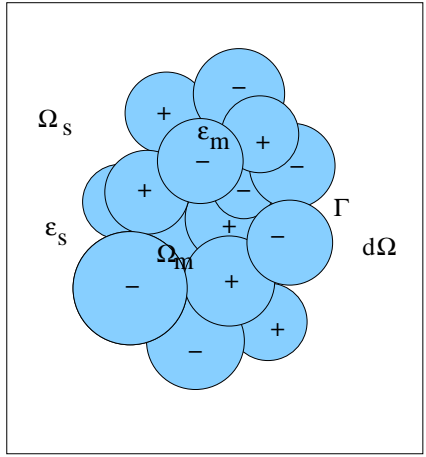} 
\caption{\small Solute and solvent regions $\Omega_m$ and $\Omega_s$ in the computational domain 
$\Omega$ for the Poisson-Boltzmann model. }
\label{fig:Protein}
\end{figure}

In electrostatic calculations the radial function is given by the 
Newton kernel (\ref{eq:Newton}), such that the free space electrostatic potential  
$P_0(x)$ in (\ref{eqn:PotSum1})  satisfies the potential equation 
 \[
  -\Delta P_0(x) = {\sum}_{\nu=1}^{N} {q_\nu}\, \delta({\|x-x_\nu\|}) \quad \mbox{in} \quad \mathbb{R}^3,
 \]
with highly singular \emph{multiparticle Dirac delta} in the right-hand side.
The latter equation may be incorporated into the more complicated elliptic equation with 
jumping coefficients across the irregular interface and with possible nonlinearity
(see Fig. \ref{fig:Protein} and the related example of the Poisson-Boltzmann equation 
in \S\ref{ssec:PBE_Applic1} also considered in \cite{BeKhKhKwSt:18}).
The RS decomposition of the multiparticle Dirac delta introduced in this paper
provides the efficient way for regularization of the solution scheme for the 
arising elliptic boundary value problem. 
 
The rest of the paper is organized as follows. 
In Section \ref{sec:RS_survey}, we recall the main ingredients for the grid-based 
tensor representation of the multivariate interaction potentials. The brief introduction to 
the RS tensor formats is also provided.
Section \ref{sec:RS_2_PBE} describes the main ingredients in the grid-based RS tensor 
splitting of the Dirac delta. We present several examples of elliptic PDEs 
where the RS tensor can  be applied to both the fundamental 
solutions and to the corresponding operator dependent Dirac delta. 
In Section \ref{sec:RS_Resolvent} we discuss the closely related concept of the RS splitting 
for the elliptic operator inverse and provide several examples demonstrating the application of 
this new approach.
In Conclusions, we summarize the main features of the proposed numerical scheme and 
overview  some possible application.

\section{Tensor representation of multiparticle interaction potentials}
\label{sec:RS_survey}

Recent tensor-structured numerical methods appeared as a bridging of the
the approximation theory for the multidimensional functions and operators
\cite{HaKhtens:04I,GHK:05,Khor1:06} with the
rank-structured tensor formats from multilinear algebra \cite{SmBroGe-Book:04}. 
One of the important results in these developments was the efficient
method for summation of the electrostatic potentials in molecular systems.
The starting point was the grid based computation of the nuclear potential operator
for molecules in the framework of the Hartree-Fock equation \cite{KhorVBAndrae:11,VeKh_box:13},
where the canonical
tensor representation of the Newton kernel was applied and the controllable accuracy
is provided due to large 3D Cartesian grids.

The tensor-structured representation of the Coulomb potential became especially
advantageous in grid-based assembled summation of the long-range potentials on large lattices.
This method was introduced in \cite{VeBokh_CPC:14}, where it was proven that the
tensor rank of a collective electrostatic potential of an arbitrarily large 3D lattice
is the same as the canonical rank of a single Newton kernel. 
For lattices of non-rectangular shapes or in presence of defects the
tensor ranks was shown to increase only slightly \cite{VeBokh_NLAA:16}.

\subsection{Rank-structured tensor representation of radial functions}
 \label{ssec:Coulomb_radial}
 
First, we recall the grid-based method for the low-rank canonical  
representation of a spherically symmetric kernel function $p(\|x\|)$, 
$x\in \mathbb{R}^d$ for $d=2,3,\ldots$, by its projection onto the set
of piecewise constant basis functions. In the case of Newton and Yukawa Green's kernels
\begin{equation}\label{eqn:NewtYukaw}
p(\|x\|)=\frac{1}{4 \pi \|x\|},\quad  \mbox{  and  }\quad
p(\|x\|)=\frac{e^{-\lambda \|x\|}}{4 \pi \|x\|},\quad
 x\in \mathbb{R}^3
\end{equation}
the numerical scheme for their representation on a fine 3D $n\times n\times n$ Cartesian grid
in the form low-rank canonical tensor was described in \cite{Khor1:06,BeKh:08}.
The corresponding approximation theory for the class of reference potentials 
like $p(\|x\|)$ in (\ref{eqn:NewtYukaw}) 
was presented in \cite{HaKhtens:04I,Khor1:06,KhorBook:17}.

In what follows, for the ease of exposition, we confine ourselves to the case $d=3$ though that 
the sinc-quadrature based separable approximation of the radial functions apply for the arbitrary 
dimension $d$.
In the computational domain  $\Omega=[-b,b]^3$, 
let us introduce the uniform $n \times n \times n$ rectangular Cartesian grid $\Omega_{n}$
with mesh size $h=2b/n$ ($n$ even).
Let $\{ \psi_\textbf{i}\}$ be a set of tensor-product piecewise constant basis functions, 
\begin{equation}\label{eqn:basis}
\psi_\textbf{i}(\textbf{x})=\prod_{\ell=1}^3 \psi_{i_\ell}^{(\ell)}(x_\ell),
\end{equation}
for the $3$-tuple index ${\bf i}=(i_1,i_2,i_3)$, $i_\ell \in I_\ell=\{1,...,n\}$, $\ell=1,\, 2,\, 3 $.
The generating kernel $p(\|x\|)$ is discretized by its projection onto the basis 
set $\{ \psi_\textbf{i}\}$
in the form of a third order tensor of size $n\times n \times n$, defined entry-wise as
\begin{equation}  \label{eqn:galten}
% {\bf P} := [p_{i_1  i_2  i_3}] \in \mathbb{R}^{n \times n \times n },
% \quad
\mathbf{P}:=[p_{\bf i}] \in \mathbb{R}^{n\times n \times n},  \quad
%\mathbf{P}:=\pc{p_\tb{i}}_{\tb{i} \in \mathcal{I}}\in \mathbb{R}^{n\times n \times n},  \quad
 p_{\bf i} = 
%p_{i_1  i_2  i_3}=
\int_{\mathbb{R}^3} \psi_{\bf i} ({x}) p(\|{x}\|) \,\, \mathrm{d}{x}.
% \frac{\psi_{i_1 i_2 i_3}(\textbf{x})}{\|\mathbf{x}\|} \,\, \mathrm{d}\tb{x}.
 % \quad \mbox{where}\quad \Omega_\tb{i}=\operatorname{supp}( \psi_\tb{i}).
\end{equation}

The low-rank canonical decomposition of the $3$rd order tensor $\mathbf{P}$ is based 
on using exponentially convergent 
$\operatorname*{sinc}$-quadratures for approximating the Laplace-Gauss transform 
to the analytic function $p(z)$, $z \in \mathbb{C}$, specified by a certain weight $\widehat{p}(t) >0$,
\begin{align} \label{eqn:laplace} 
p(z)=\int_{\mathbb{R}_+} \widehat{p}(t) e^{- t^2 z^2} \,\mathrm{d}t \approx
\sum_{k=-M}^{M} p_k e^{- t_k^2 z^2} \quad \mbox{for} \quad |z| > 0,\quad z \in \mathbb{R},
\end{align} 
with the proper choice of the quadrature points $t_k$ and weights $p_k$.
Under the assumption $0< a \leq |z |  < \infty$
this quadrature can be proven to provide an exponential convergence rate in $M$
for a class of analytic functions $p(z)$. 
 The $sinc$-quadrature based approximation to generating function by 
using the short-term Gaussian sums in (\ref{eqn:laplace}), (\ref{eqn:hM}) are applicable 
to the class of analytic functions
in certain strip $|z|\leq D $ in the complex plane, such that on the real axis these functions decay
polynomially or exponentially. We refer to basic results in 
\cite{Stenger,Braess:BookApTh,HaKhtens:04I}, 
where the exponential convergence of the $sinc$-approximation in the number of terms 
(i.e., the canonical rank) was analyzed. 
Now, for any fixed $x=(x_1,x_2,x_3)\in \mathbb{R}^3$, 
such that $\|{x}\| > a > 0$, %$\|\mathbf{x}\|$
we apply the $\operatorname*{sinc}$-quadrature approximation (\ref{eqn:laplace}), (\ref{eqn:hM}) 
to obtain the separable expansion
\begin{equation} \label{eqn:sinc_Newt}
 p({\|{x}\|}) =   \int_{\mathbb{R}_+} \widehat{p}(t)
e^{- t^2\|{x}\|^2} \,\mathrm{d}t  \approx 
\sum_{k=-M}^{M} p_k e^{- t_k^2\|{x}\|^2}= 
\sum_{k=-M}^{M} p_k  \prod_{\ell=1}^3 e^{-t_k^2 x_\ell^2},
\end{equation}
providing an exponential convergence rate in $M$,
% Under the assumption $0< a \leq \|{x}\| \leq A < \infty$
% this approximation provides the exponential convergence rate in $M$,
\begin{equation} \label{eqn:sinc_conv}
\left|p({\|{x}\|}) - \sum_{k=-M}^{M} p_k e^{- t_k^2\|{x}\|^2} \right|  
\le \frac{C}{a}\, \displaystyle{e}^{-\beta \sqrt{M}},  
\quad \text{with some} \ C,\beta >0.
\end{equation}

Combining \eqref{eqn:galten} and \eqref{eqn:sinc_Newt}, and taking into account the 
separability of the Gaussian basis functions, we arrive at the low-rank 
approximation to each entry of the tensor $\mathbf{P}$,
\begin{equation*} \label{eqn:C_nD_0}
 p_{\bf i} \approx \sum_{k=-M}^{M} p_k   \int_{\mathbb{R}^3}
 \psi_{\bf i}({x}) e^{- t_k^2\|{x}\|^2} \mathrm{d}{x}
=  \sum_{k=-M}^{M} p_k  \prod_{\ell=1}^{3}  \int_{\mathbb{R}}
\psi^{(\ell)}_{i_\ell}(x_\ell) e^{- t_k^2 x^2_\ell } \mathrm{d} x_\ell.
\end{equation*}
Define the vector (recall that $p_k >0$) 
\begin{equation} \label{eqn:galten_int}
\textbf{p}^{(\ell)}_k
= p_k^{1/3} \left[b^{(\ell)}_{i_\ell}(t_k)\right]_{i_\ell=1}^{n_\ell} \in \mathbb{R}^{n_\ell}
\quad \text{with } \quad b^{(\ell)}_{i_\ell}(t_k)= 
\int_{\mathbb{R}} \psi^{(\ell)}_{i_\ell}(x_\ell) e^{- t_k^2 x^2_\ell } \mathrm{d}x_\ell,
\end{equation}
then the $3$rd order tensor $\mathbf{P}$ can be approximated by 
the $R$-term ($R=2M+1$) canonical representation
\begin{equation} \label{eqn:sinc_general}
    \mathbf{P} \approx  \mathbf{P}_R =
\sum_{k=-M}^{M} p_k \bigotimes_{\ell=1}^{3}  {\bf b}^{(\ell)}(t_k)
= \sum\limits_{k=-M}^{M} {\bf p}^{(1)}_k \otimes {\bf p}^{(2)}_k \otimes {\bf p}^{(3)}_k
\in \mathbb{R}^{n\times n \times n}, % \quad {\bf p}^{(\ell)}_k \in \mathbb{R}^n.
%\quad a_k,\, t_k \in \mathbb{R},
\end{equation}
where ${\bf p}^{(\ell)}_k \in \mathbb{R}^n$. 
Given a threshold $\varepsilon >0 $,  $M$ can be chosen as the minimal number
such that in the max-norm
\begin{equation*} \label{eqn:error_control}
\| \mathbf{P} - \mathbf{P}_R \|  \le \varepsilon \| \mathbf{P}\|.
\end{equation*}
%where $\varepsilon >0 $ is the given tolerance. 
The skeleton vectors can be re-numerated by $k \mapsto q=k+M+1$, 
${\bf p}^{(\ell)}_k \mapsto {\bf p}^{(\ell)}_{q}$, ($q=1,...,R$), $\ell=1, 2, 3$.
The canonical tensor ${\bf P}_{R}$ in (\ref{eqn:sinc_general})
approximates the 3D symmetric kernel function 
$p({\|x\|})$ ($x\in \Omega$), centered at the origin, such that 
${\bf p}^{(1)}_q={\bf p}^{(2)}_q={\bf p}^{(3)}_q$ ($q=1,...,R$).

In the case of Newton kernel we have $p(z)=1/z$, $\widehat{p}(t)=\frac{2}{\sqrt{\pi}}$, and
the Laplace-Gauss transform representation reads
\[
 \frac{1}{z}= \frac{2}{\sqrt{\pi}}\int_{\mathbb{R}_+} e^{- z^2 t^2 } dt, \quad 
 \mbox{where}\quad z=\|x\|, \quad x \in \mathbb{R}^3
 %z=\sqrt{x_1^2 + x_2^2  + x_3^2}.
\]
that can be rewritten  the form (by using the substitution $t=\log (1+ e^u)$)
\[
 \frac{1}{z}=\frac{2}{\sqrt{\pi}} \int_{\mathbb{R}} f_1(u,z) du, \quad \mbox{with} \quad 
 f_1(u,z)= \frac{e^{-z^2 \log^2 (1+ e^{u})}}{ (1+ e^{-u})}.
\]
Then the second substitution leads to the  integral representation
\[
 \frac{1}{z}=\frac{2}{\sqrt{\pi}} \int_{\mathbb{R}} F(w,z) dw, \quad 
 F(w,z):= \mbox{cosh}(w) f_1(\mbox{sinh}(w),z),
\]
such that the set of quadrature points and discretized integrand  can be chosen  by 
\begin{equation} \label{eqn:hM}
w_k=k \mathfrak{h}_M , \quad F_k=F(w_k,z) \mathfrak{h}_M, \quad 
\mathfrak{h}_M=C_0/\sqrt{M} , \quad C_0>0,
\end{equation}
providing the exponential convergence rate as in (\ref{eqn:sinc_conv}).

One can observe form numerical tests that there are canonical  vectors representing the long- 
and short-range (highly localized) contributions to the total electrostatic potential. 
This interesting feature was also recognized for the rank-structured 
tensors representing a lattice sum of electrostatic potentials 
\cite{VeBokh_CPC:14,VeBokh_NLAA:16}.

\subsection{Tensor splitting of the kernel into long- and short-range parts}\label{ssec:S_L_split}

 From the definition of the quadrature (\ref{eqn:sinc_general}), (\ref{eqn:hM}), we can easily observe
that the full set of approximating Gaussians includes two classes of functions: those with 
small "effective support" and the long-range functions. Consequently, functions from different classes 
may require different tensor-based schemes for their efficient numerical treatment.
Hence, the idea of the new approach is the constructive implementation of a 
range separation scheme that allows the independent efficient treatment 
of both the long- and short-range parts in each summand in (\ref{eqn:PotSum1}).

In what follows, without loss of generality, we confine ourselves to the case of 
the Newton kernel in $\mathbb{R}^3$, 
so that the sum in (\ref{eqn:sinc_general}) via (\ref{eqn:hM}) reduces to $k=0,1,\ldots,M$ 
(due to symmetry argument).
 From (\ref{eqn:hM}) we observe that the sequence of quadrature points $\{t_k\}$,   
can be split into two subsequences, 
$
{\cal T}:=\{t_k|k=0,1,\ldots,M\}={\cal T}_l \cup {\cal T}_s,
$ 
with
\begin{equation} \label{eqn:Split_Qpoints}
{\cal T}_l:=\{t_k\,|k=0,1,\ldots,R_l\}, \quad  
\mbox{and} \quad {\cal T}_s:=\{t_k\,|k=R_l+1,\ldots,M\}.
\end{equation}
Here
${\cal T}_l$ includes quadrature points $t_k$ condensed ``near'' zero, hence generating 
the long-range Gaussians (low-pass filters), 
and ${\cal T}_s$ accumulates the increasing in $M\to \infty$ 
sequence of ``large'' sampling points $t_k$ with the upper bound $C_0^2 \log^2(M)$, 
corresponding to the short-range Gaussians (high-pass filters).
Notice that the quasi-optimal choice of the constant $C_0\approx 3$ was determined in \cite{BeKh:08}.
We further denote ${\cal K}_l:=\{k\,| k=0,1,\ldots,R_l\}$ and ${\cal K}_s:=\{k\,| k=R_l+1,\ldots,M \}$. 

Splitting (\ref{eqn:Split_Qpoints}) generates the additive decomposition of the canonical tensor
$\mathbf{P}_R$ onto the short- and long-range parts,
\[
 \mathbf{P}_R = \mathbf{P}_{R_s} + \mathbf{P}_{R_l},
\]
where
\begin{equation} \label{eqn:Split_Tens}
    \mathbf{P}_{R_s} =
\sum\limits_{t_k\in {\cal T}_s} {\bf p}^{(1)}_k \otimes {\bf p}^{(2)}_k \otimes {\bf p}^{(3)}_k, 
\quad \mathbf{P}_{R_l} =
\sum\limits_{t_k\in {\cal T}_l} {\bf p}^{(1)}_k \otimes {\bf p}^{(2)}_k \otimes {\bf p}^{(3)}_k.
\end{equation}

The choice of the critical number $R_l=\# {\cal T}_l -1$ (or equivalently, $R_s=\# {\cal T}_s = M-R_l$), 
that specifies the splitting ${\cal T}={\cal T}_l \cup {\cal T}_s$,
is determined by the \emph{active support} %$\sigma >0$ 
of the short-range components 
such that one can cut off the functions ${\bf p}_k(x)$, $t_k\in {\cal T}_s$, 
outside of the sphere $B_{\sigma}$ of radius ${\sigma}>0$, 
subject to a certain threshold $\delta>0$.
For fixed $\delta>0$, the choice of $R_s$ is uniquely defined by 
the (small) parameter $\sigma$ and vise versa.  Given $\sigma$,
the following two basic criteria, corresponding to (A) the max- and (B) $L^1$-norm estimates 
can be applied:
\begin{equation} \label{eqn:Split_crit_max}
 (A) \quad {\cal T}_s=\{t_k:\,a_k e^{-t_k^2 \sigma^2} \leq \delta \}\;\Leftrightarrow \;
 R_l=\min k : \,  a_k e^{-t_k^2 \sigma^2} \leq \delta, 
\end{equation}
\begin{equation} \label{eqn:Split_crit_L2}
 (B)\quad {\cal T}_s:=\{t_k:\, a_k \int_{B_{\sigma}} e^{-t_k^2 x^2} dx \leq \delta\}\;
 \Leftrightarrow \; R_l=\min k : \,
  a_k \int_{B_{\sigma}} e^{-t_k^2 x^2} dx \leq \delta.
\end{equation}
The quantitative estimates on the value of $R_l$ can be easily calculated by using the explicit equation
(\ref{eqn:hM}) for the quadrature parameters. 
For example, in case $C_0=3$ and $a(t)=1$, criteria (A) implies that $R_l$ solves the equation
\[
 \left(\frac{3R_l \log M}{M}\right)^2 \sigma^2 = \log(\frac{\mathfrak{h}_M}{\delta}).
\]
Criteria (\ref{eqn:Split_crit_max}) and (\ref{eqn:Split_crit_L2}) can be modified 
depending on the particular applications.
For example, in electronic structure calculations, the parameter $\sigma$ can be associated with the 
typical inter-atomic distance in the molecular system of interest.

\begin{figure}[htb]
\centering
 \includegraphics[width=7.5cm]{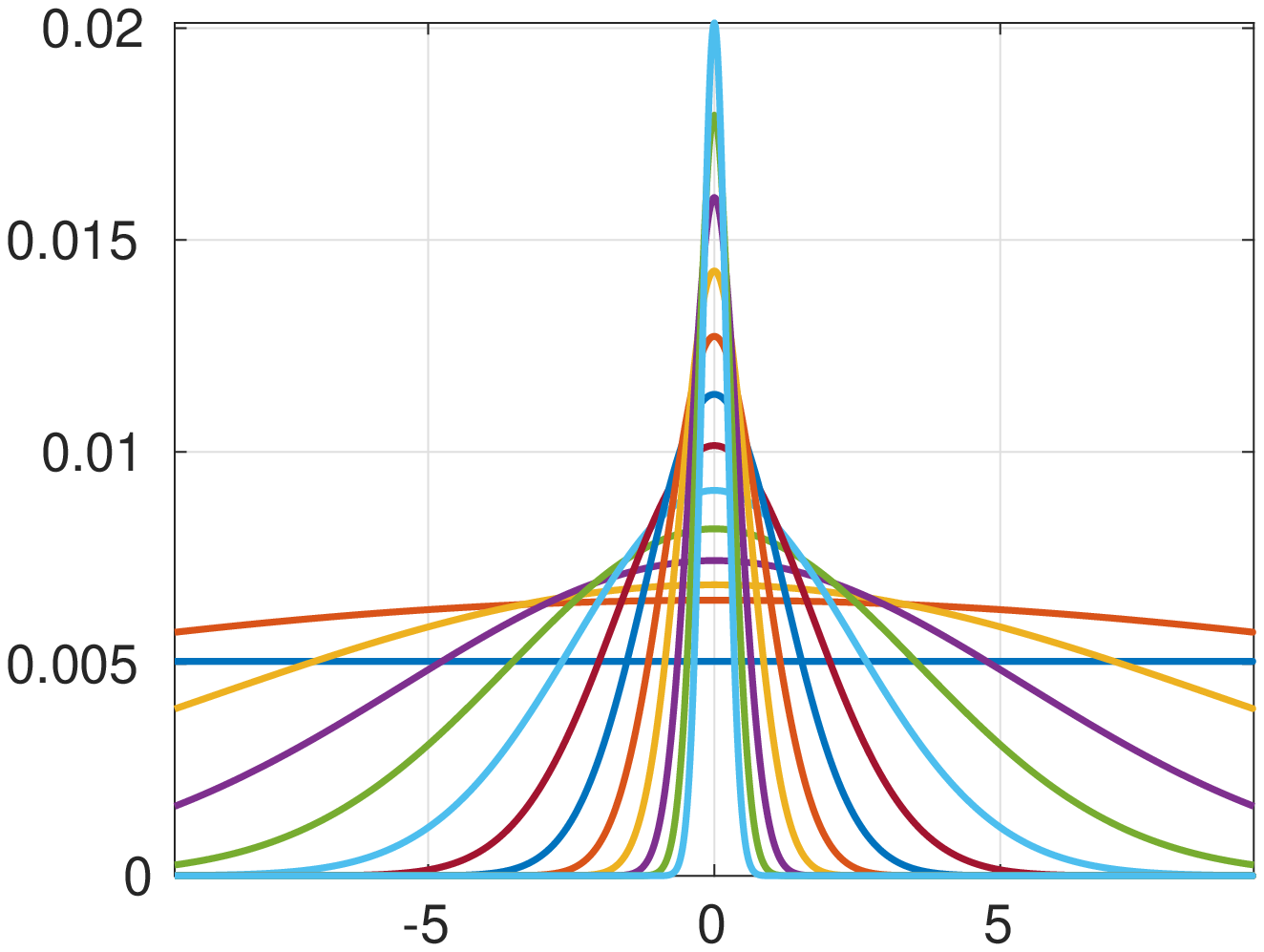} 
\includegraphics[width=7.6cm]{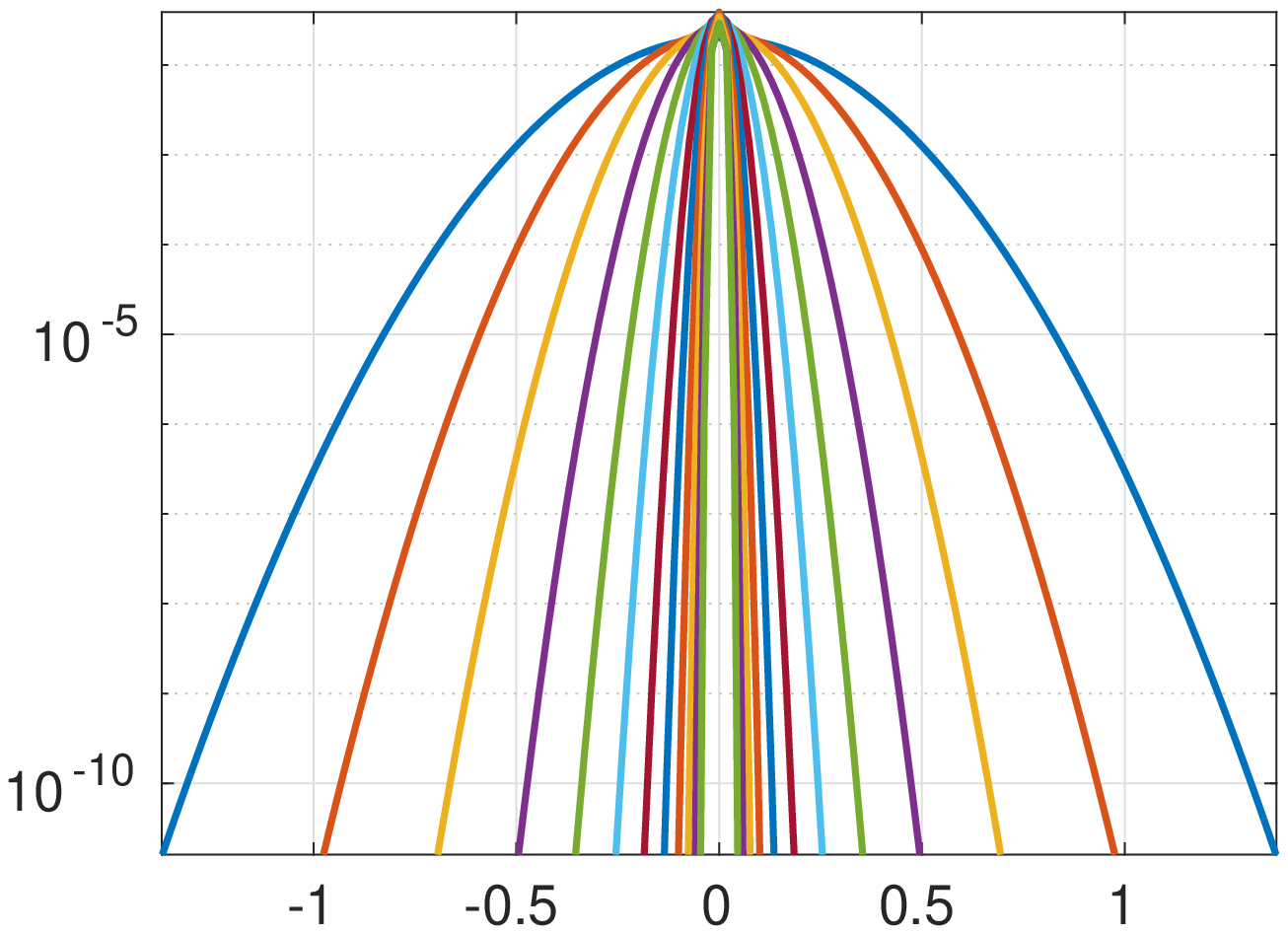}
\caption{Long-range vectors (left) and and short-range (in logarithmic scale) 
canonical vectors for $n=1024$, $R=24, RL=13$. }
\label{fig:1024_rl12}
\end{figure}

Figure \ref{fig:1024_rl12}   % and \ref{fig:1024_rs8}  
illustrates the splitting (\ref{eqn:Split_Qpoints}) for the tensor 
${\bf P}_R= {\bf P}_{R_l} + {\bf P}_{R_s}$ in (\ref{eqn:Split_Tens})
represented on 
the $n\times n\times n$ grid with the parameters $R=20, R_l=12$ and $R_s=8$, 
respectively (cf. \cite{BeKhKh_RS_SISC:18}).
Following criteria (A) with $\delta \approx 10^{-4}$, the effective support 
for this splitting is determined by $\sigma =0.9$.
The complete Newton kernel depicted in Figure \ref{fig:New_andDelt}, left,
exhibits the long-range behavior, 
while the function values of the tensor ${\bf P}_{R_s}$ decay exponentially fast 
apart of the effective support, see Figure \ref{fig:1024_rl12}, right.
 
Inspection of the quadrature point distribution in (\ref{eqn:hM}) shows that the 
short- and long-range subsequences are nearly equally balanced distributed, 
so that one can expect approximately 
\begin{equation} \label{eqn:ChoiceR_l}
 R_s \approx R_l=M/2.
\end{equation} 
The optimal choice may depend on the particular  kernel function  and applications.

The advantage of the range separation in the splitting of 
the canonical tensor 
$${\bf P}_R \mapsto {\bf P}_{R_s} + {\bf P}_{R_l}$$  
in  (\ref{eqn:Split_Tens}) is the opportunity for 
independent tensor representations of both sub-tensors ${\bf P}_{R_s}$ and ${\bf P}_{R_l}$ 
providing the separate treatment of the short- and long-range parts in the total sum of 
many pointwise interaction potentials as in (\ref{eqn:PotSum1}).

 \subsection{Brief introduction to the RS tensor format}   \label{ssec:RS_format}
 
 The novel range separated (RS) tensor format has been recently introduced in 
 \cite{BeKhKh_RS:16,BeKhKh_RS_SISC:18}.
 This format is well suited for modeling of
 the long-range interaction potential in multi-particle systems. 
  It is based on the partitioning of the reference tensor representation of the newton kernel
 into long-range and short-range part. 
 
%  From (\ref{eqn:hM}) we observe that the sequence of quadrature points $\{t_k\}$,   
% can be split into two subsequences, 
% $
% {\cal T}:=\{t_k|k=0,1,\ldots,M\}={\cal T}_l \cup {\cal T}_s,
% $ 
% with
% \begin{equation} \label{eqn:Split_Qpoints}
% {\cal T}_l:=\{t_k\,|k=0,1,\ldots,R_l\}, \quad  
% \mbox{and} \quad {\cal T}_s:=\{t_k\,|k=R_l+1,\ldots,M\}.
% \end{equation}
% Here
% ${\cal T}_l$ includes quadrature points $t_k$ condensed ``near'' zero, hence generating 
% the long-range Gaussians (low-pass filters), 
% and ${\cal T}_s$ accumulates the increasing in $M\to \infty$ 
% sequence of ``large'' sampling points $t_k$ with the upper bound $C_0^2 \log^2(M)$, 
% corresponding to the short-range Gaussians (high-pass filters).
%  
 
According to the tensor canonical representation of the Newton kernel (\ref{eqn:sinc_general}) 
as a sum of Gaussians, one can distinguish their supports into the short- and long-range parts,
$
 \mathbf{P}_R = \mathbf{P}_{R_s} + \mathbf{P}_{R_l},
$
given by (\ref{eqn:Split_Tens}).
% where
% \begin{equation} \label{eqn:Split_Tens}
%     \mathbf{P}_{R_s} =
% \sum\limits_{t_k\in {\cal T}_s} {\bf p}^{(1)}_k \otimes {\bf p}^{(2)}_k \otimes {\bf p}^{(3)}_k, 
% \quad \mathbf{P}_{R_l} =
% \sum\limits_{t_k\in {\cal T}_l} {\bf p}^{(1)}_k \otimes {\bf p}^{(2)}_k \otimes {\bf p}^{(3)}_k.
% \end{equation}
Then the RS splitting (\ref{eqn:Split_Tens}) is 
applied to the reference canonical tensor ${\bf P}_R$ 
and to its accompanying version $\widetilde{\bf P}_R=[\widetilde{p}_R(i_1,i_2,i_3)]$, 
$i_\ell \in \widetilde{I}_\ell$, $\ell=1,2,3$ living on the double size grid with the 
same mesh size, such that
\[
 \widetilde{\bf P}_R = \widetilde{\mathbf{P}}_{R_s} + \widetilde{\mathbf{P}}_{R_l} 
 \in \mathbb{R}^{2n \times 2n \times 2n}.
\]

  The total electrostatic potential $P_0(x)$ in (\ref{eqn:PotSum1}) is represented by a projected tensor 
${\bf P}_0\in \mathbb{R}^{n \times n \times n}$ that can 
be constructed by a direct sum of shift-and-windowing transforms of the reference 
tensor $\widetilde{\bf P}_R$ (see \cite{VeBokh_CPC:14} for more details),
\begin{equation}\label{eqn:Total_Sum}
 {\bf P}_0 = \sum_{\nu=1}^{N} {z_\nu}\, {\cal W}_\nu (\widetilde{\bf P}_R)=
 \sum_{\nu=1}^{N} {z_\nu} \, {\cal W}_\nu (\widetilde{\mathbf{P}}_{R_s} + \widetilde{\mathbf{P}}_{R_l})
 =: {\bf P}_s + {\bf P}_l.
\end{equation}
The shift-and-windowing transform ${\cal W}_\nu$ maps a reference tensor 
$\widetilde{\bf P}_R\in \mathbb{R}^{2n \times 2n \times 2n}$ onto its sub-tensor 
of smaller size $n \times n \times n$, obtained by first shifting the center of
the reference tensor $\widetilde{\bf P}_R$ to the grid-point $x_\nu$ and then restricting 
(windowing) the result onto the computational grid $\Omega_n$.
  
The difficulty of the tensor representation   (\ref{eqn:Total_Sum})  is
  that the number of terms in the canonical representation of the full tensor sum ${\bf P}_0$
increases almost proportionally to the number $N$ of particles in the system.
And it is not tractable for the rank reduction \cite{BeKhKh_RS:16}.

This problem is solved in \cite{BeKhKh_RS:16,BeKhKh_RS_SISC:18} by considering  the global tensor 
decomposition of only the  "long-range part" in the tensor ${\bf P}_0$, defined by
\begin{equation}\label{eqn:Long-Range_Sum} %{eqn:Total_Sum}
 {\bf P}_l = \sum_{\nu=1}^{N} {z_\nu} \, {\cal W}_\nu (\widetilde{\mathbf{P}}_{R_l})=
 \sum_{\nu=1}^{N} {z_\nu} \, {\cal W}_\nu 
 (\sum\limits_{k\in {\cal K}_l} \widetilde{\bf p}^{(1)}_k \otimes \widetilde{\bf p}^{(2)}_k 
 \otimes \widetilde{\bf p}^{(3)}_k).
 \end{equation}
 And for tensor representation of the sums of short-range parts, a sum of cumulative 
 tensors of small support (and small size) is used, accomplished by the list of the 3D
 potentials coordinates.
 
 In what follows, we recall the definition of the RS tensor format
 \begin{definition}\label{Def:RS-Can_format} (RS-canonical tensors \cite{BeKhKh_RS:16,BeKhKh_RS_SISC:18}). 
 Given the separation parameter $\gamma \in \mathbb{N}$ and a set of points $x_\nu \in \mathbb{R}^{d}$,
 $\nu=1,\ldots,N$,
 the RS-canonical tensor format specifies the class of $d$-tensors 
 ${\bf A}  \in \mathbb{R}^{n_1\times \cdots \times n_d}$
 which can be represented as a sum of a rank-${R}_L$ canonical tensor  
 \[
{\bf A}_{R_L} = {\sum}_{k =1}^{R_L} \xi_k {\bf a}_k^{(1)} \otimes \cdots \otimes {\bf a}_k^{(d)}
\in \mathbb{R}^{n_1\times ... \times n_d}
\]
and a cumulated canonical tensor 
 \[
 \widehat{\bf A}_S={\sum}_{\nu =1}^{N} c_\nu {\bf A}_\nu ,   
\]
generated by replication of the reference tensor ${\bf A}_0$ to the points $x_\nu$.
Then the RS canonical tensor is presented as
\begin{equation}\label{eqn:RS_Can}
 {\bf A} =  {\bf A}_{R_L} + \widehat{\bf A}_S=
 {\sum}_{k =1}^{R_L} \xi_k {\bf a}_k^{(1)}  \otimes \cdots \otimes {\bf a}_k^{(d)} +
 {\sum}_{\nu =1}^{N} c_\nu {\bf A}_\nu, 
\end{equation}
where $\mbox{rank}({\bf A}_0)\leq R_0$ and $\mbox{diam}(\mbox{supp}{\bf U}_0)\leq 2 \gamma$ in the index size.
\end{definition}
The storage size for  the  RS-canonical tensor ${\bf A}$ in (\ref{eqn:RS_Can}) 
is estimated by (\cite{BeKhKh_RS_SISC:18}, Lemma 3.9),
$$
\mbox{stor}({\bf A})\leq d R n + (d+1)N + d R_0 \gamma.
$$
%Given ${\bf i}\in {\cal I}$, 
Denote by $\overline{\bf u}^{(\ell)}_{i_\ell}\in \mathbb{R}^{1\times R}$ the  
row-vector  with index $i_\ell$ in the side matrix $U^{(\ell)}\in \mathbb{R}^{n_\ell \times R}$
of ${\bf U}$, and by $\xi=(\xi_1,\ldots,\xi_d)$ the coefficient vector in (\ref{eqn:RS_Can}).
Then the ${\bf i}$-th entry of the RS-canonical tensor ${\bf A}=[a_{\bf i}]$ can be calculated 
as a sum of long- and short-range contributions by
\[
 a_{\bf i}= \left(\odot_{\ell=1}^d \overline{\bf u}^{(\ell)}_{i_\ell} \right) \xi^T +
 \sum_{\nu \in {\cal L}({\bf i})} c_\nu {\bf U}_\nu({\bf i}),
 \quad \mbox{at the expense}\quad O(d R + 2 d \gamma R_0).
\]
%at the expense  $O(d R + 2 d \gamma R_0)$.
 
 In application to the calculation of multi-particle interaction potentials  discussed above
 we associate the tensors ${\bf P}_s$ and ${\bf P}_l$ in (\ref{eqn:Total_Sum}) 
 with short- and long-range components ${\bf A}_{R_L}$ and  $\widehat{\bf A}_S$ in the RS 
 representation of the collective electrostatic potential ${\bf P}_0$.
 The following theorem proofs the almost uniform in $N$ bound on the Tucker (and canonical) 
 rank of the tensor ${\bf A}_{R_L}={\bf P}_l$, representing the long-range part of ${\bf P}_0$.
  \begin{theorem}\label{thm:Rank_LongRange}
 (Uniform rank bounds for the long-range part,  \cite{BeKhKh_RS:16,BeKhKh_RS_SISC:18}).
Let the long-range part ${\bf P}_l$ in the total interaction potential, see (\ref{eqn:Long-Range_Sum}),
correspond to the choice of splitting parameter 
in (\ref{eqn:ChoiceR_l}) with $M=O(\log^2\varepsilon)$.
Then the total $\varepsilon$-rank ${\bf r}_0$ of the Tucker approximation to the canonical tensor sum ${\bf P}_l$
is bounded by
\begin{equation}\label{eqn:Rank_LongR}
 |{\bf r}_0|:=rank_{Tuck}({\bf P}_l)=C\, b \,\log^{3/2} (|\log (\varepsilon/N)|),
\end{equation}
where the constant $C$ does not depend on the number of particles $N$.
\end{theorem}

 \section{Operator dependent RS tensor decomposition of the $\delta$-function }\label{sec:RS_2_PBE}

In this section, we propose the new splitting scheme for the elliptic problem 
which is based on the range separated 
representation of the $\delta$-function which enters the highly singular 
right-hand side in the target elliptic PDE, say,  for the Poisson-Boltzmann model (\ref{eqn:PBE})
considered below.

\subsection{Overlook on explicit approximations of the Dirac delta}\label{sec:Dirac_loc_appr}

The commonly used approximation of the $\delta$-function can be constructed in 
the form of Gaussian $C(\lambda)\mbox{e}^{- \lambda^2\|x\|^2}$ as $\lambda \to \infty$.
Let $d=3$, the harmonic potential of the $3$-dimensional Gaussian  
$f(x) = \mbox{e}^{- \lambda^2\|x\|^2}$ is represented by using the $\mbox{erf}$-function
\[
 {\cal L}^{-1} (\mbox{e}^{- \lambda^2\|\cdot\|^2}) (x) = 
 \frac{1}{4\pi}\int_{\mathbb{R}^3} \frac{\mbox{e}^{- \lambda^2\|y\|^2}}{\|x-y\|} d y= 
 \frac{1}{2 \|\lambda x\|}\int_{0}^{\|\lambda x\|} \mbox{e}^{-t^2}\, dt= 
 \frac{\sqrt{\pi}}{4 \lambda \|x\|}\mbox{erf}(\|\lambda x\|).
\]
This provides the corresponding approximant to the Green kernel, which, however does not 
lead to the range-separated representation in question.

Notice that the alternative non-smooth approximation to the Dirac delta by the step-type bumps
 \[
  \delta_h(x) := h^{-d}, \quad \mbox{if}\quad |x_\ell|\leq h, \; \ell=1,\ldots,d; \quad 0 \;\; \mbox{otherwise},
 \]
does not allow to catch the short- and long-range parts of the Green kernels of interest.

For the range-separated tensor representation, we apply the  ``clever'' constructed 
weighted sum of Gaussian-type functions $\Sigma_M= \sum_{k=1}^M p_k e^{\lambda_k \|x\|^2}$
 (see \S\ref{ssec:Coulomb_radial} on the sinc quadrature based approach)
 to approximate the Green kernel and, respectively, the associated Dirac delta. 
 The corresponding parameters $\lambda_k$ vary from very small to rather large values.
 It can be seen that the harmonic images of Gaussians with  $\lambda_k > 1 $ and $\lambda_k < 1 $
 exhibit quite different asymptotic behavior, hence, making the construction of an 
 approximant $\Sigma_M$ as rather nontrivial task. In this way, the sinc quadrature approximation in 
 \S\ref{ssec:Coulomb_radial} does a good job.

For evaluation of the co-normal derivatives on interfaces 
(with unit normal vector ${\bf n}$), represented by 
$\nabla \cdot {\bf n}$, one also needs to calculate  the gradients of 
the harmonic potential of the Gaussian.
To that end the gradient $\nabla {\cal L}^{-1} (\mbox{e}^{- \|\cdot\|^2})$ of the harmonic 
potential of the Gaussian in 
$\mathbb{R}^d$ can be calculated by using the expression for partial derivatives 
\[
 \frac{\partial}{\partial x_\ell} {\cal L}^{-1} (\mbox{e}^{- \lambda^2\|\cdot\|^2}) (x) =
 - \frac{x_\ell}{2 \|x\|^2}
 \left( \frac{\sqrt{\pi}}{2 \|\lambda^2 x\|}\mbox{erf}(\|\lambda x\|) - \mbox{e}^{- \lambda^2 \|x\|^2} \right), 
 \quad \ell=1,\ldots,d.
\]
Now we look for application of Laplacain to $d$-dimensional Gaussian $\mbox{e}^{- \lambda^2 \|x \|^2}$,
which leads to the radial function
\begin{equation} \label{eqn:Lapl_on Gaussian}
 -\Delta_d (\mbox{e}^{- \lambda \|x \|^2}) = 
(2d \lambda - 4 \lambda^2 \|x \|^2)\mbox{e}^{- \lambda \|x \|^2}:= G_d(\|x \|).
 %L_1^{-1/2}(\|x \|^2).
\end{equation}
This means that Gaussian is the harmonic potential of $G_d$,
\[
 \mbox{e}^{- \lambda \|x \|^2} = {\cal L}^{-1} G_d(\|\cdot\|).
\]
The behavior of the radial function 
\[
G_d(r)= (2d \lambda  - 4 \lambda^2 r^2)\mbox{e}^{- \lambda r^2}, \quad r=\|x \|, r\geq 0
\]
depends on the parameter $\lambda >0$, such that we have 
\[
G_d(0) = 2d \lambda, \quad G_d(r_\ast)=0\quad \mbox{for} \quad 
r_\ast = \sqrt{\frac{d}{2 \lambda}}
\]
and $G_d'(r_0)= 0$ for $r_0>0$ satisfying the equation
\[
 r_0^2 = \frac{2+d}{2\lambda}.
\]
Hence we obtain 
$$
G_d(r_0)= (2d \lambda - 4 \lambda^2 r_0^2)\mbox{e}^{- \lambda r_0^2}= 
 % (2d \lambda - 4 \lambda^2 \frac{2+d}{2\lambda})\mbox{e}^{- \lambda \frac{2+d}{2\lambda}}=
-4\lambda \mbox{e}^{-\frac{2+d}{2}}
$$ 
that determines  the maximum values of $|G_d(r)|$ on the interval $r\in [r_\ast, \infty)$.

\begin{figure}[htb]
\centering
\includegraphics[width=7.5cm]{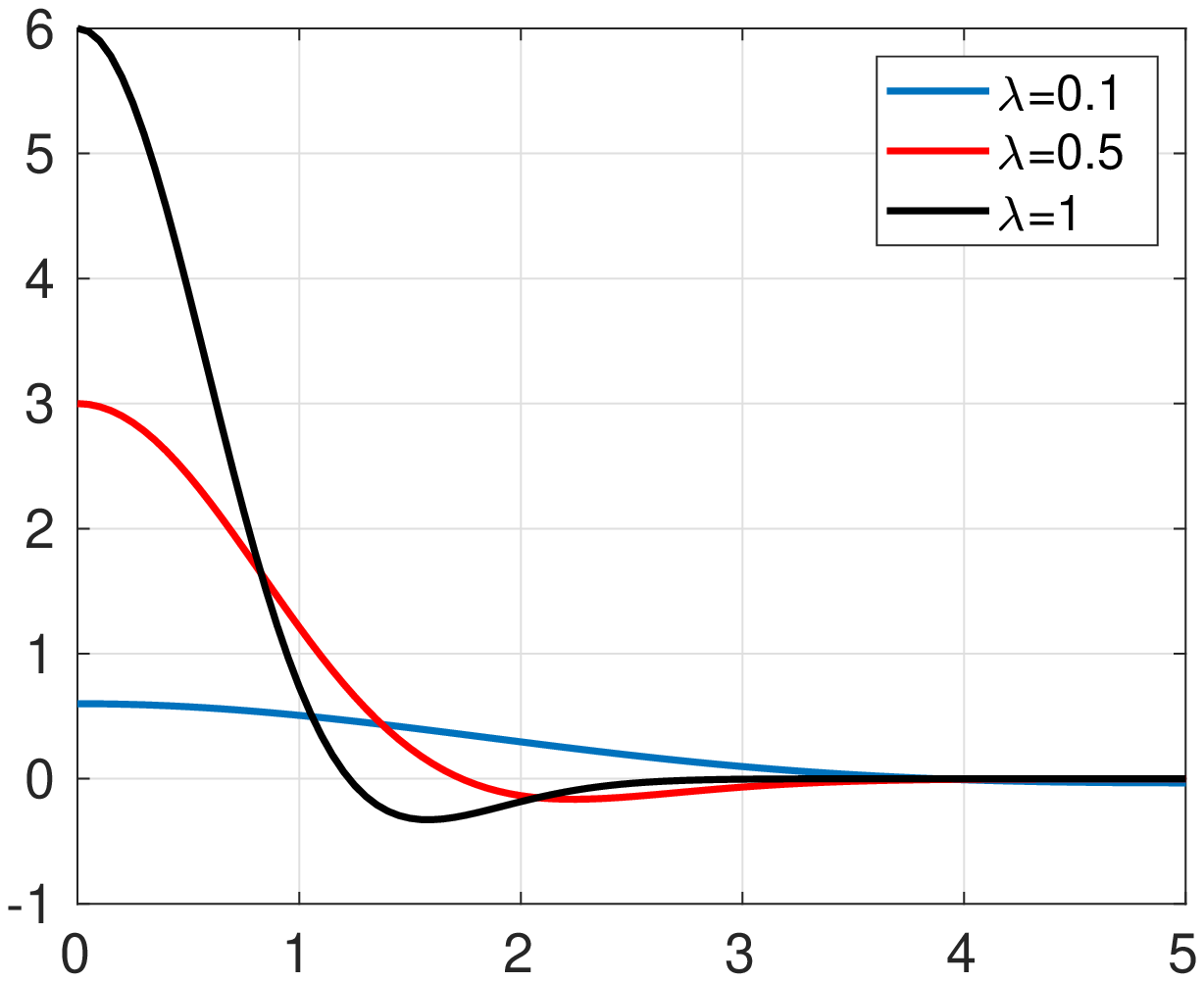}
 \includegraphics[width=7.5cm]{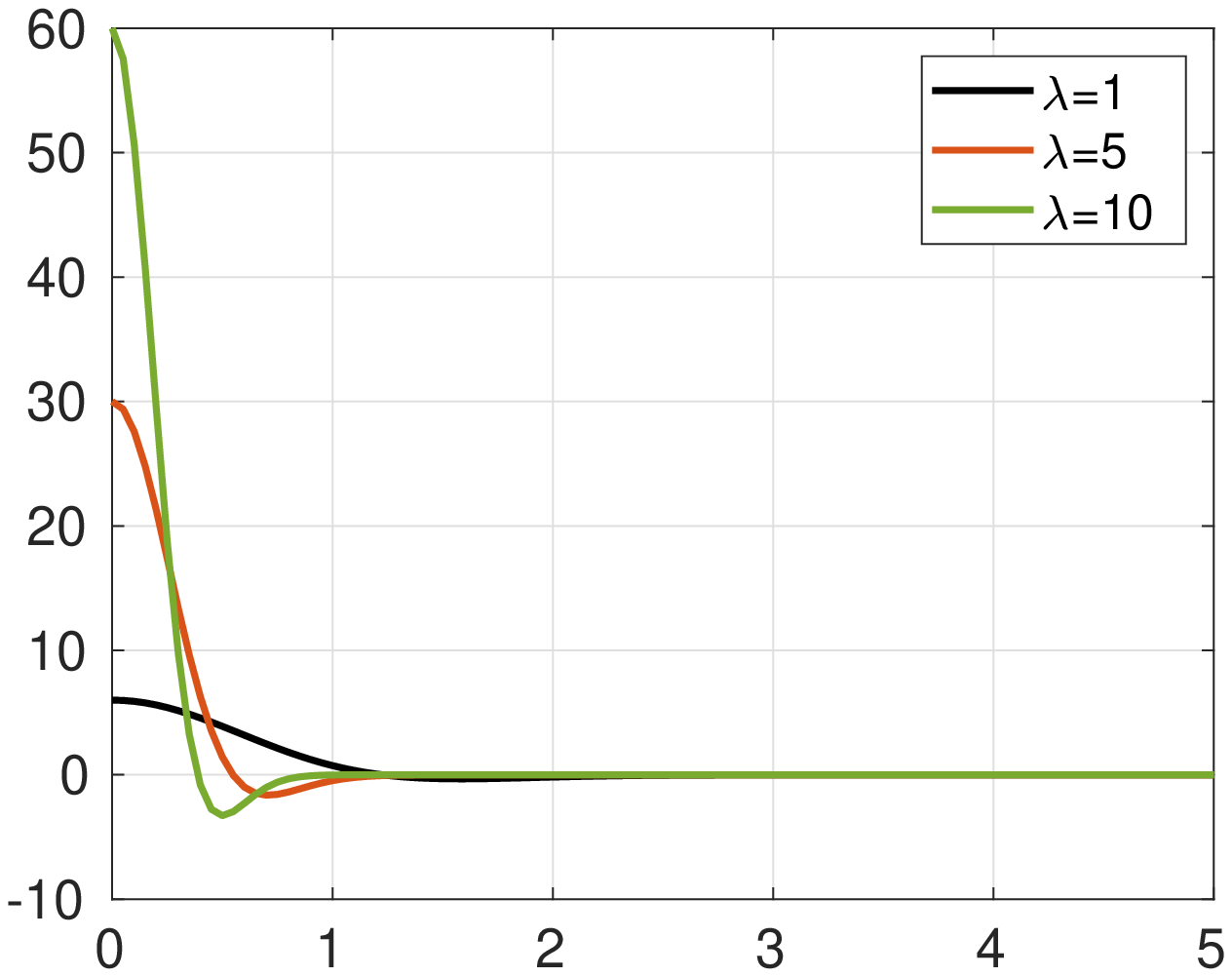}
\caption{\small Function $G_3(r)$ visualized for different values of the parameter $\lambda\in [0.1,10]$.}
\label{fig:Lapl_Gaussian}
\end{figure}
Figure \ref{fig:Lapl_Gaussian} visualizes the share of function $G_d(r)$ in
(\ref{eqn:Lapl_on Gaussian}) for $d=3$ depending on different parameters $\lambda\in [0.1,10]$.
We can observe the expected behavior: Gaussian with large exponents $\lambda_k$ in the RS decomposition
of the discretized reference Newton kernel $\mathbf{P}_R$ are responsible for the representation of a 
``singular'' part in the Dirac delta, while the terms with small $\lambda_k$  serve for recovering
the ``nonlocal'' part, that should vanish after full summation over indices $k=1,\ldots,M$.

\subsection{RS decomposition of the discretized Dirac delta}\label{sec:RS_Dirac_discr}

The idea of the RS decomposition is to modify the right-hand side $\rho_f$
in such a way that the short-range part in the solution $u$ can be computed independently 
and initial equation applies only to its long-range part. The latter is a smooth function, hence
the FEM approximation error can be reduced dramatically even on the relatively small grids in 3D.

To fix the idea, we consider the simplest case of the single atom with charge $q=1$ located at the 
origin, such that $u(x)=\frac{1}{\|x\|}$, $x\in \mathbb{R}^3$. 
Recall that
the Newton kernel (\ref{eqn:sinc_general}) discretized by the $R$-term sum of Gaussians 
living on the tensor grid
$\Omega_n$, is represented by a sum of the short- and long-range tensors,
\[
 \frac{1}{\|x\|}   \rightsquigarrow \mathbf{P}_R = \mathbf{P}_{R_s} + \mathbf{P}_{R_l},
\]
where
\begin{equation} \label{eqn:Split_Tens1}
    \mathbf{P}_{R_s} =
\sum\limits_{k=R_l+1}^R {\bf p}^{(1)}_k \otimes {\bf p}^{(2)}_k \otimes {\bf p}^{(3)}_k, 
\quad \mathbf{P}_{R_l} =
\sum\limits_{k=1}^{R_l} {\bf p}^{(1)}_k \otimes {\bf p}^{(2)}_k \otimes {\bf p}^{(3)}_k.
\end{equation}

Let us consider the formal discrete version of the exact equation for the potential
\[
 -\Delta \frac{1}{4 \pi \|x\|} =\delta(x)
\]
by its FD analogy by substitution of $\mathbf{P}_R $ instead of $u(x)$
and FEM Laplacian matrix $A_{\Delta}$ instead of $\Delta$. This lead to the grid representation 
of the Dirac delta
\[
 \delta(x) \rightsquigarrow \boldsymbol{\delta}_h:=-A_{\Delta} \mathbf{P}_R,
\]
which is associated with its differential representation. 

Recall  that in the case  $d=3$ the FEM Laplacian matrix $A_{\Delta}$ takes a form
\begin{equation}\label{eqn:Lapl_Kron3}
A_{\Delta} = \Delta_{1} \otimes I_2\otimes I_3 + I_1 \otimes \Delta_{2} \otimes I_3 + 
I_1 \otimes I_2\otimes \Delta_{3},
\end{equation}
where $-\Delta_\ell = \frac{1}{h^2}\mathrm{tridiag} \{ 1,-2,1 \} \in \mathbb{R}^{n_\ell\times n_\ell}$,
$\ell=1,2,3$, 
denotes the discrete univariate Laplacian, and $I_\ell\in \mathbb{R}^{n_\ell\times n_\ell}$  is the identity 
matrix.
Here the Kronecker rank  of $A_{\Delta}$ equals to $3$.

Then the canonical tensor representation $\boldsymbol{\delta}_h$ of the ${\delta}$-function approximated on 
$n\times n \times n$ Cartesian grid can be computed as the action of the 
Laplace operator on the Newton kernel given in the canonical rank-$R$ tensor format as 
follows\footnote{The tensor-based scheme for evaluation of the Laplace operator
in an arbitrary separable basis set was described in \cite{KhorVBAndrae:11}
for calculation of the kinetic energy of electrons in electronic structure calculations.}
\[
 \boldsymbol{\delta}_h = -A_{\Delta} \mathbf{P}_R=
 \sum\limits_{k=1}^R  \Delta_1 {\bf p}^{(1)}_k \otimes {\bf p}^{(2)}_k \otimes {\bf p}^{(3)}_k 
 + \sum\limits_{k=1}^R {\bf p}^{(1)}_k \otimes \Delta_2 {\bf p}^{(2)}_k \otimes {\bf p}^{(3)}_k 
 + \sum\limits_{k=1}^R {\bf p}^{(1)}_k \otimes  {\bf p}^{(2)}_k \otimes \Delta_3 {\bf p}^{(3)}_k. 
\]
The above representation was also applied in \cite{BeKhKhKwSt:18}.
Figure \ref{fig:New_andDelt} shows the cross-sections of the Newton kernel and of the
Dirac delta computed at the plane $z=0$ using $n\times n\times n $ 
3D Cartesian grids with $n=1024$.
\begin{figure}[htb]
\centering
\includegraphics[width=7.5cm]{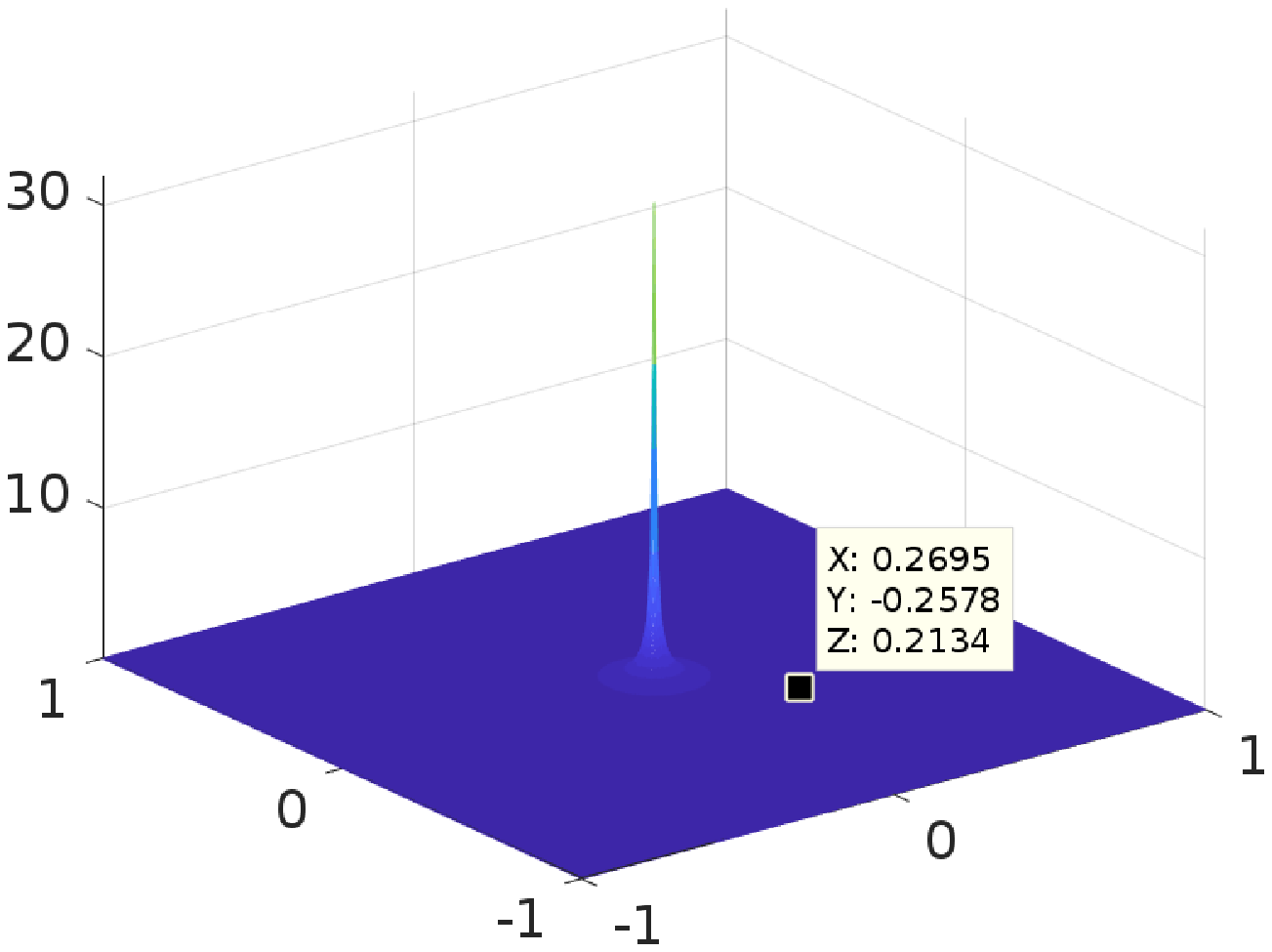}
 \includegraphics[width=7.5cm]{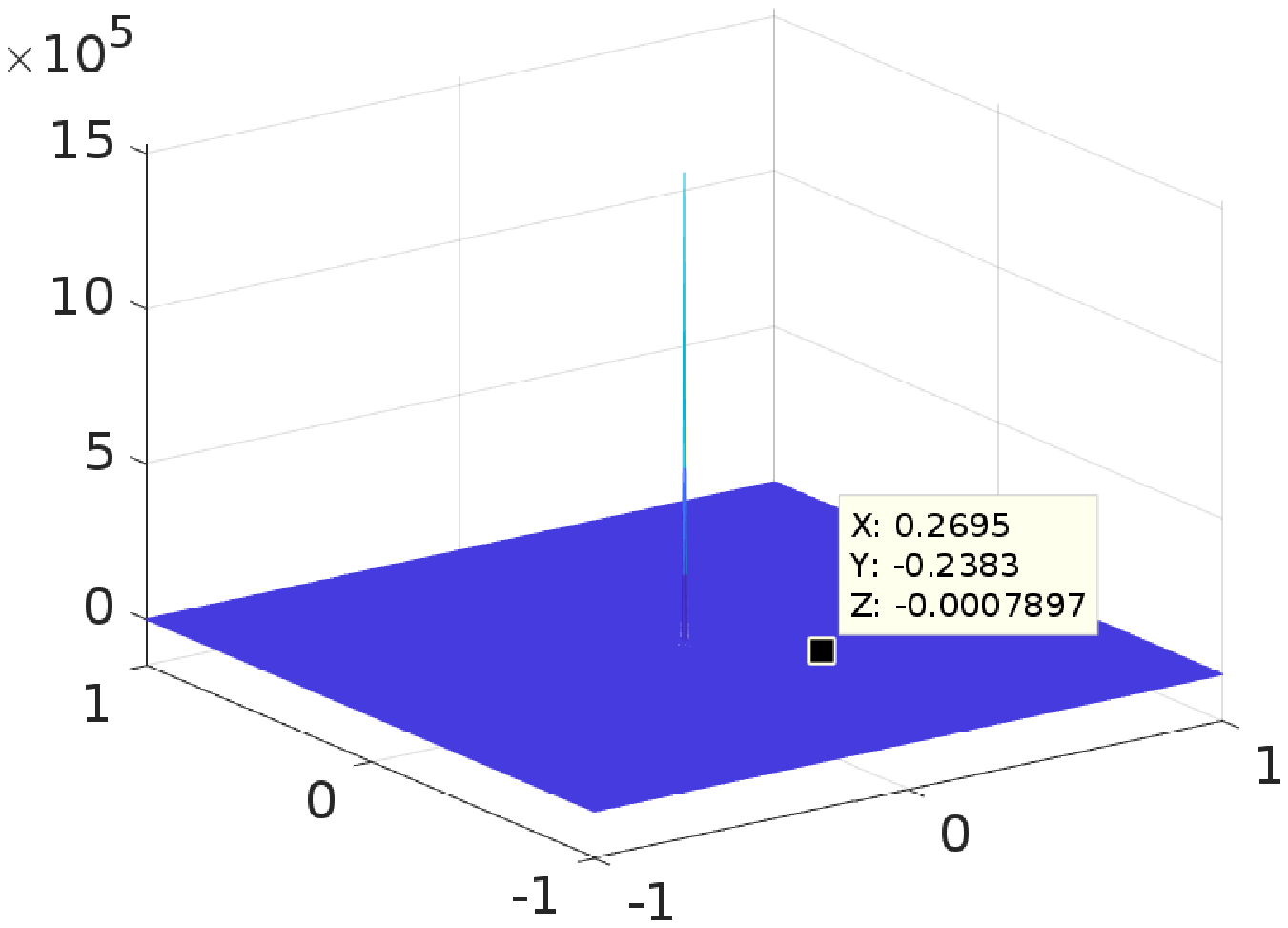} 
\caption{\small The Newton kernel (left) and the Dirac delta (right) computed using the 
canonical tensor representation on $n \times n \times n $ 3D Cartesian grid with $n=1024$, $R=20$. }
\label{fig:New_andDelt}
\end{figure}

\begin{figure}[htb]
\centering
 \includegraphics[width=7.5cm]{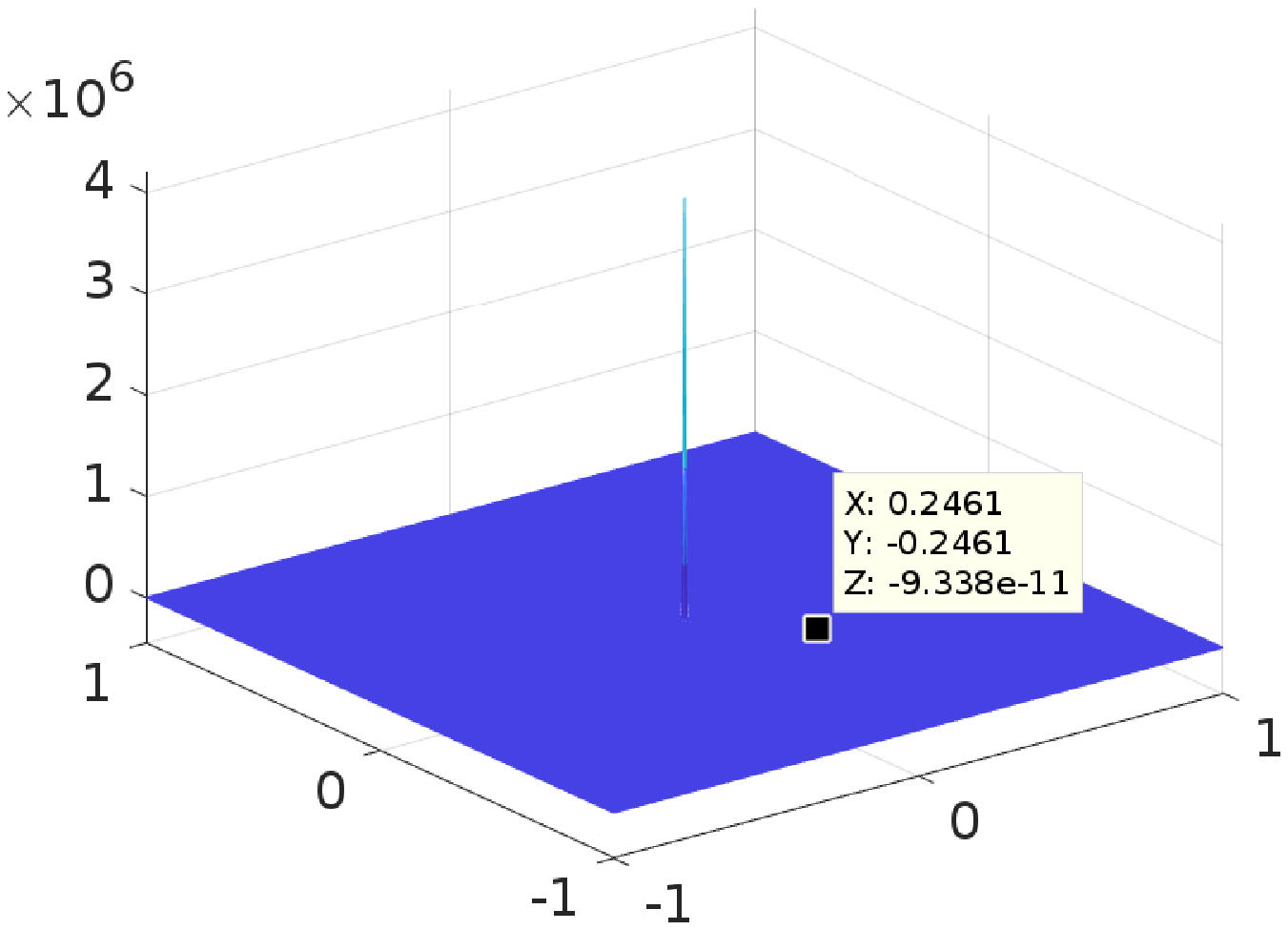} 
\includegraphics[width=7.5cm]{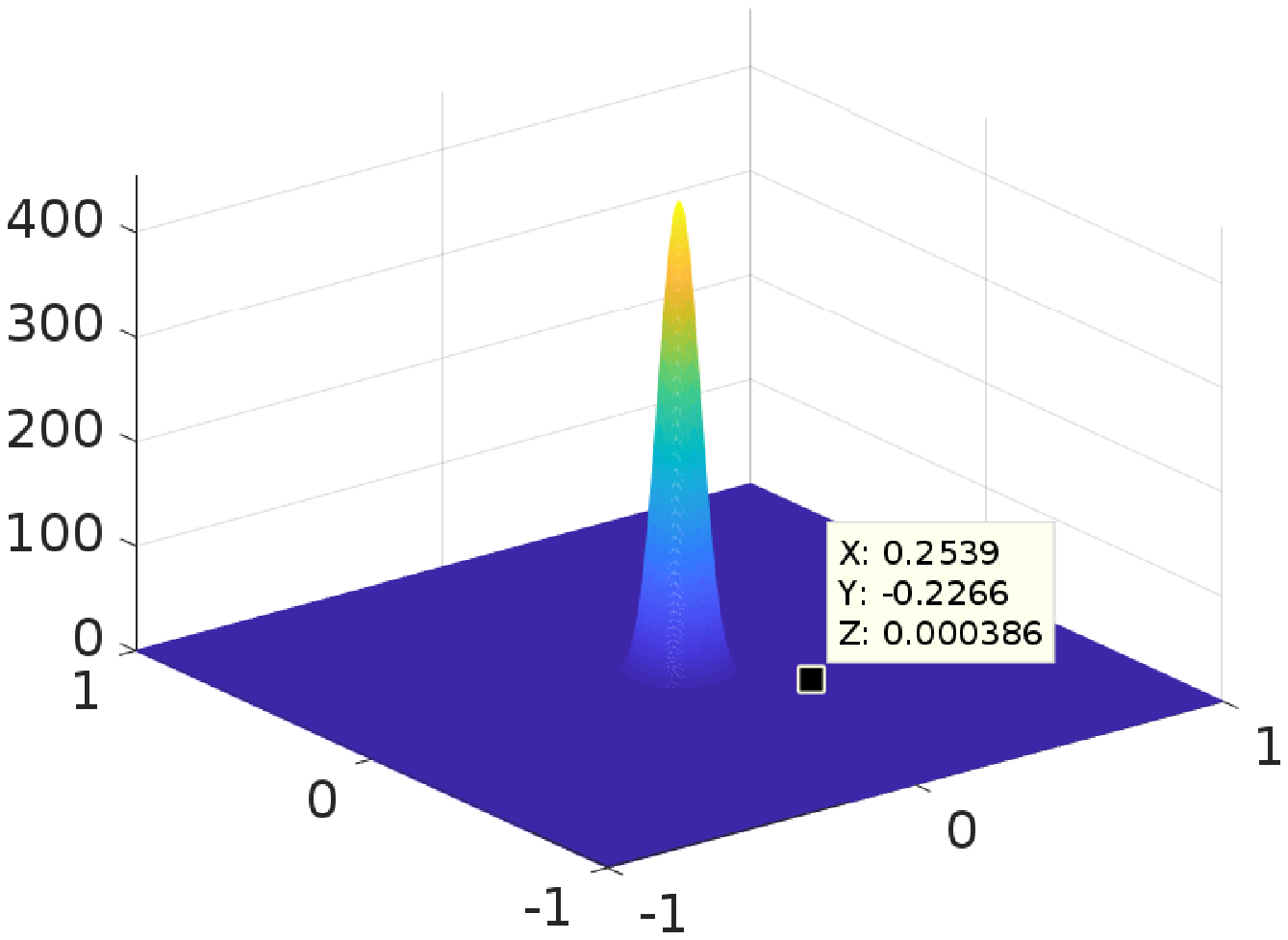}
\caption{\small Delta function computed by using only short-range (left) and only long-range (right) 
canonical vectors for $n=1024$, $R=20, R_l=12$. }
\label{fig:ShortLong_Delt}
\end{figure}

 Now we are in a position to construct the RS tensor based splitting scheme.
 To that end, we introduce the short- and long-range splitting of the 
 discretized $\boldsymbol{\delta}_h$ function  in the form
 \[
  \boldsymbol{\delta}_h= \boldsymbol{\delta}_s + \boldsymbol{\delta}_l,
 \]
where
\[
 \boldsymbol{\delta}_s:= -A_{\Delta} \mathbf{P}_{R_s}, \quad 
 \mbox{and} \quad \boldsymbol{\delta}_l:= -A_{\Delta} \mathbf{P}_{R_l}.
\]
Now we observe that by definition the short range part vanishes on the interface $\Gamma$,
hence it satisfies the discrete Poisson equation in $\Omega_m$ with the respective right-hand
side in the form $\boldsymbol{\delta}_s$ and zero boundary conditions on $\Gamma$. 
Then we deduce that this equation can be subtracted from the full
discrete linear system, such that the long-range component of the solution, $\mathbf{P}_{R_l}$,
will satisfy to the same linear system of equations (same interface conditions), 
but with the right-hand side 
corresponding to the weighted sum of the long-range tensors  $\boldsymbol{\delta}_l$ only.
In our simple example, we arrive at the particular equation for the long-range part in 
$\mathbf{P}_R$,
\[
 - A_{\Delta} {\bf u}_l = \boldsymbol{\delta}_l,
\]
which, by the construction, recovers the long-range part in the free space harmonic potential.
 
 \begin{figure}[htb]
\centering
\includegraphics[width=7.5cm]{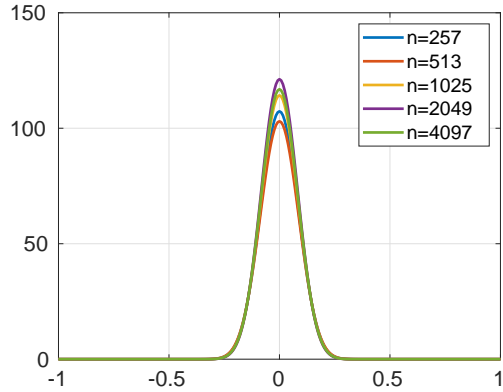}
\caption{\small The long-range part of the Dirac delta  computed using the 
canonical tensor representation on $n \times n \times n $ 3D Cartesian grid with $n=257,...,4097$, 
$R=20$. }
\label{fig:Long_Delta_grids}
\end{figure} 
This scheme can be easily extended to the case of many atomic systems described by (\ref{eqn:PotSum1}).
 Let ${\bf P}_{R_L}$ be the rank-$R_L$ canonical representation of the ``long-range part''
 in the tensor ${\bf P}_{0}$ given by (\ref{eqn:Total_Sum}) and (\ref{eqn:Long-Range_Sum}), 
 obtained by the Can-to-Tuck-to-Can  compression scheme \cite{KhKh3:08}. 
 Then the tensor ${\bf P}_{R_L}$ can be approximated with controllable 
 accuracy by the solution ${\bf u}_{l,N}\approx {\bf P}_{R_L}$
 of the discrete Poisson problem with the right-hand side obtained by the $\varepsilon$-adapted
 rank compression  of the $\mbox{rank} \leq 3 R_L$ tensor $A_{\Delta} {\bf P}_{R_L}$, 
 \[
  - A_{\Delta} {\bf u}_{l,N} = \boldsymbol{\delta}_{l,N} := 
   T_\varepsilon(\sum\limits_{\nu=1}^N \boldsymbol{\delta}_{l,\nu})=
  - T_\varepsilon(A_{\Delta} {\bf P}_{R_L}), \quad 
  \boldsymbol{\delta}_{l,\nu} = -A_{\Delta} ({z_\nu} \, {\cal W}_\nu \widetilde{\bf P}_{R_l}),
 \]
 where $T_\varepsilon$ denotes the operator of the $\varepsilon$-rank truncation.
The numerical efficiency of this approach in the case of many-particle systems 
is based on the important property that 
the canonical $\varepsilon$-rank of the tensor $\boldsymbol{\delta}_{l,N}$ is almost uniformly bounded 
in the number of particles $N$ in the system, as proven by the following statement similarly to that 
in Theorem \ref{thm:Rank_LongRange}, \cite{BeKhKh_RS:16,BeKhKh_RS_SISC:18}:

\begin{lemma} \label{lem:Rank_DirDelta}
Let the long-range part ${\bf P}_l$ in the collective interaction potential, see (\ref{eqn:Long-Range_Sum}),
correspond to the choice of splitting parameter in (\ref{eqn:ChoiceR_l}) with $M=O(\log^2\varepsilon)$.
Then the Tucker $\varepsilon$-rank of the long-range tensor 
$A_{\Delta} {\bf P}_{R_L}\in \mathbb{R}^{n\times n \times n}$ is bounded by
\begin{equation}\label{eqn:Rank_LongR_Dirac}
 |{\bf r}_0|:=rank_{Tuck}(A_{\Delta} {\bf P}_{R_L})=C_3\, b \,\log^{3/2} (|\log (\varepsilon/N)|),
\end{equation}
where the constant $C_3=3 C$ does not depend on the number of particles $N$
 (here $C$ is the respective constant in (\ref{eqn:Rank_LongR})). 
The corresponding canonical rank is bounded by $|{\bf r}_0|^2$.

The tensor $A_{\Delta} {\bf P}_{R_L}$ requires $O(|{\bf r}_0|^3 + 3|{\bf r}_0|\, n)$
numbers to store.
\end{lemma}
\begin{proof}
 The proof follows from Theorem \ref{thm:Rank_LongRange} taking into account that the Kronecker rank
 of the 3D discrete Laplacian is equal to $3$. 
 The bound on the storage size follows from the rank estimate (\ref{eqn:Rank_LongR_Dirac}).
\end{proof}

\begin{figure}[htb]
\centering
\includegraphics[width=7.0cm]{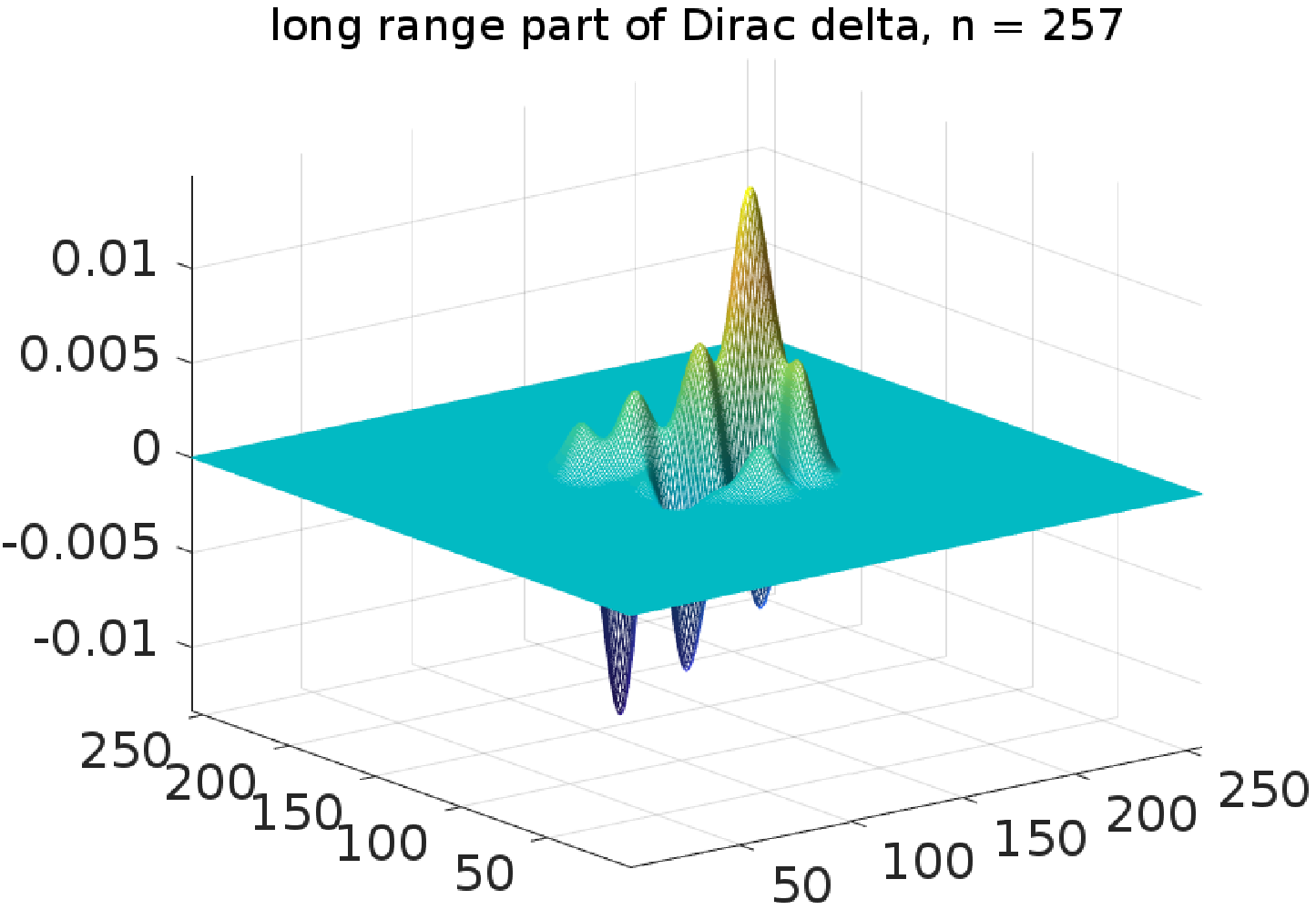}
\includegraphics[width=7.0cm]{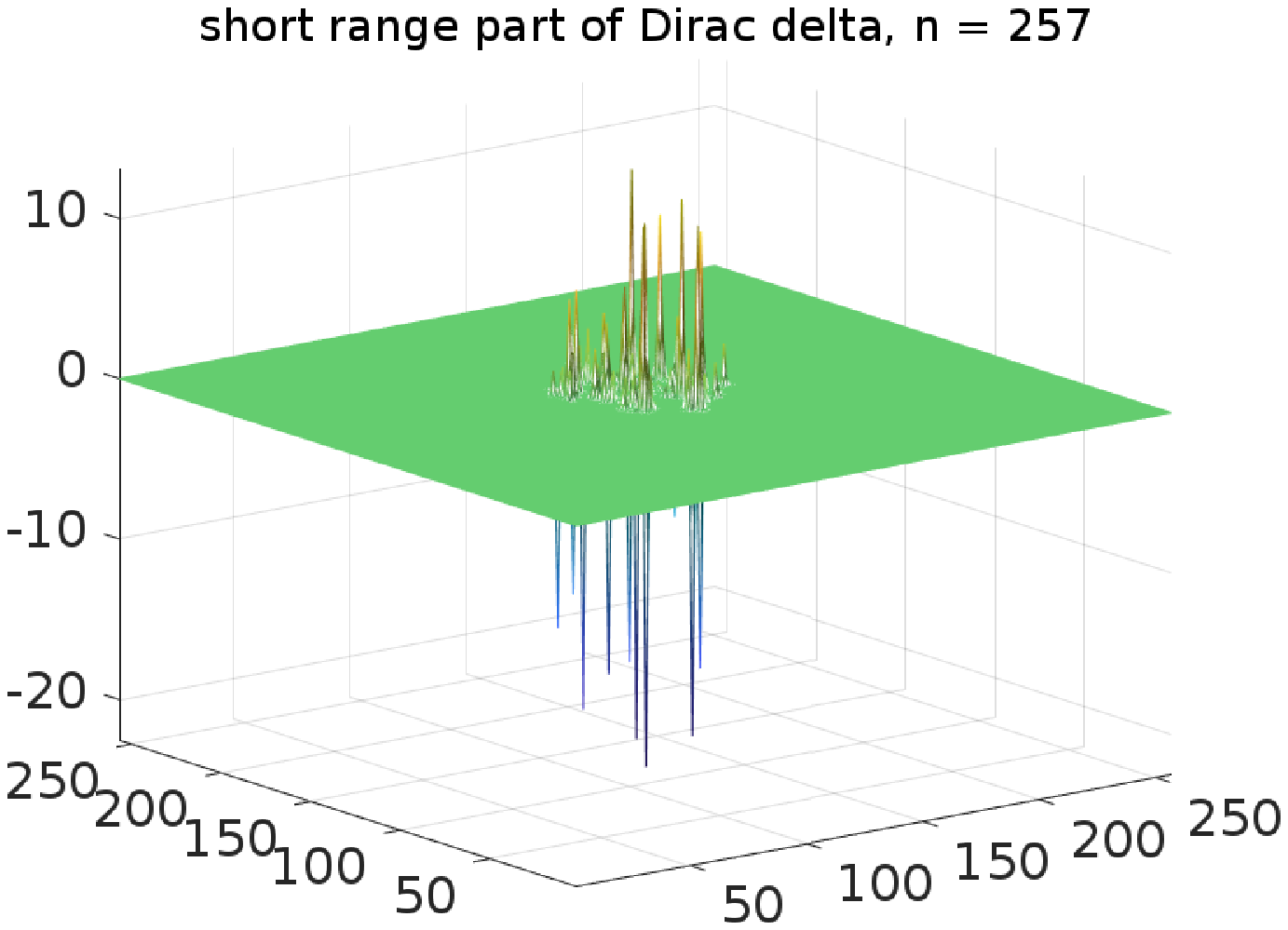}
\caption{\small The long- and short-range parts of the $N$-particle 
Dirac delta  computed using the canonical tensor representation on 
$n \times n \times n $ grid with $n=257$, and $N=782$. }
\label{fig:Delta_long_Nsum_257}
\end{figure} 
  \begin{figure}[htb]
\centering
\includegraphics[width=7.0cm]{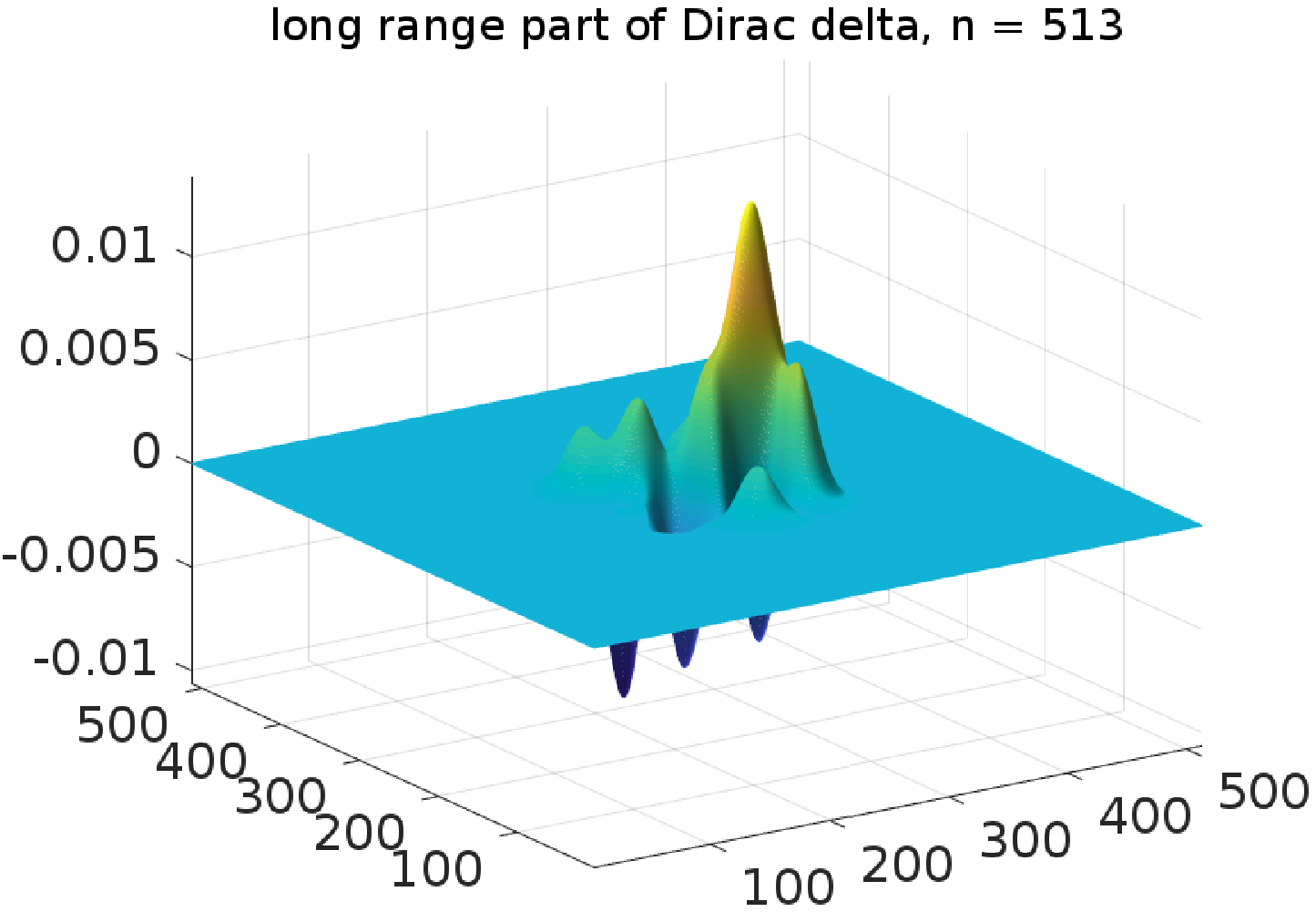}
\includegraphics[width=7.0cm]{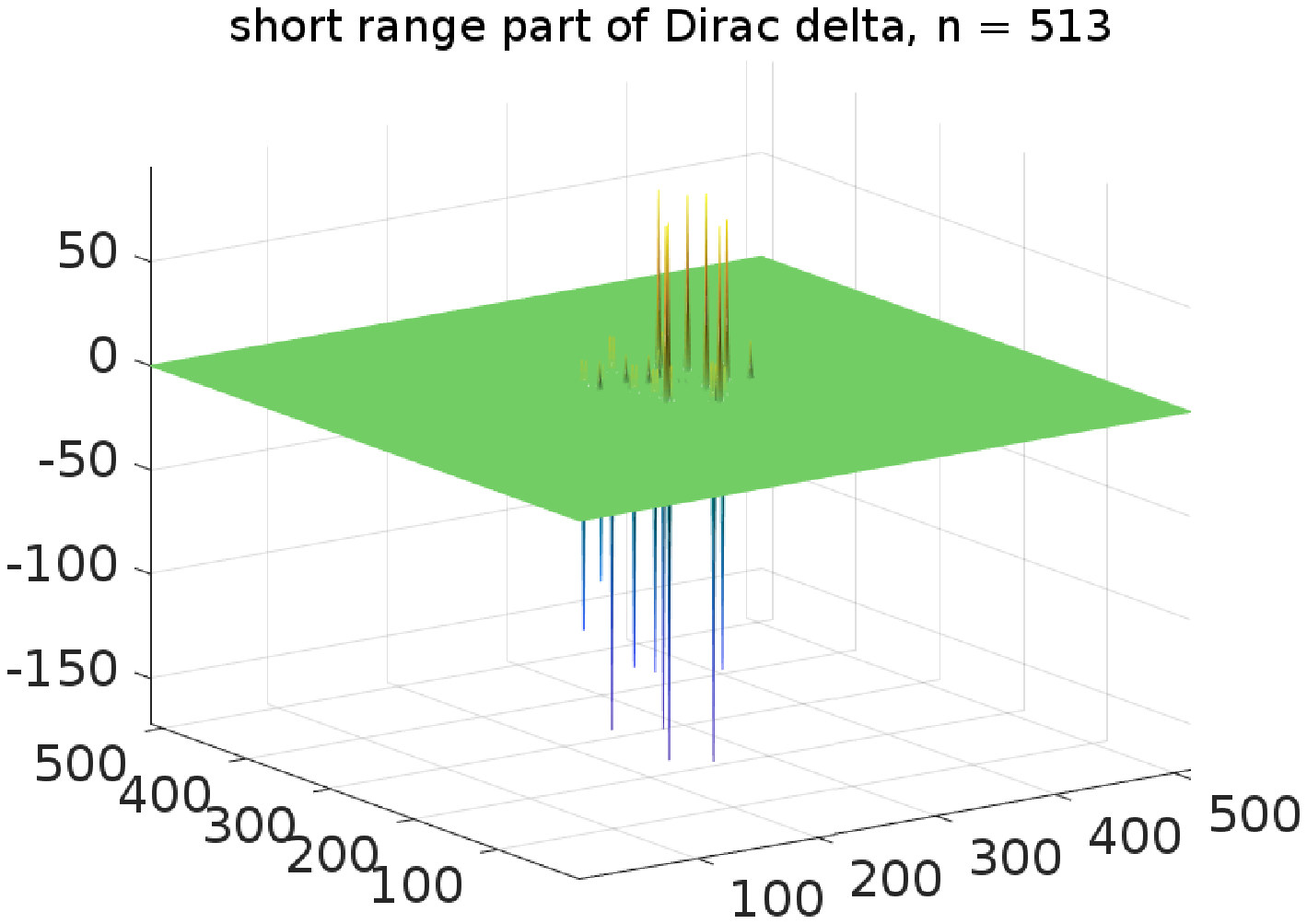}
\caption{\small The long- and short-range parts of the $N$-particle 
Dirac delta  computed using the canonical tensor representation on 
$n \times n \times n $ grid with $n=513$, $R=20$, and $N=782$.}
\label{fig:Delta_long_Nsum_513}
\end{figure}  
Figure \ref{fig:Long_Delta_grids} demonstrates the behavior of the long-range 
tensor $A_{\Delta} {\bf P}_{R_L}$
computed on a sequence of large $n \times n \times n $ 3D Cartesian grids for 
$n=257,513,\ldots,4097$.
We observe that the effective support of this tensor is getting smaller for larger values of 
the grid size $n$, which that is important for solving elliptic PDEs with heterogeneous data arising
in large scale applications 
(see also \S\ref{ssec:PBE_Applic1} -- \S\ref{ssec:PBE_Applic3}).

  Figures \ref{fig:Delta_long_Nsum_257} and \ref{fig:Delta_long_Nsum_513} represent the long- 
  and short-range components of the $N$-particle Dirac delta in the form 
  $\sum\limits_{{k}=1}^N q_k \delta(\|{x} - {x}_{k}\|)$,
  corresponding to the protein like neutral molecular system composed of $N=500$ atoms.
  Similar molecular systems were used in \cite{BeKhKh_RS:16,BeKhKh_RS_SISC:18} in numerical 
  tests on the RS tensor approximation.
  The RS tensor representation is implemented on $n \times n \times n $ Cartesian grids
  with $n=257,513$.  We observe that the long-range part remains stable for finer grid, while 
  the short-range part increases with factor for double mesh-size.
  The canonical rank of the reference potential was set up for $R=24$, with
  $R_l=12$. The Tucker rank of the collective potential is about $R_{Tuck}=30$ which allows the low rank 
  canonical decomposition of the target tensor 
  $\boldsymbol{\delta}_{l,\nu} = -A_{\Delta} ({z_\nu} {\cal W}_\nu \widetilde{\bf P}_{R_l})$ by using the 
  can-to-Tuck-to-can rank reduction algorithm \cite{KhKh3:08} applied to the initial rank-$9384$
  canonical tensor.

\subsection{The linearized Poisson-Boltzmann model }\label{ssec:PBE_Applic1}

The Poisson-Boltzmann equation (PBE) is the
commonly used model for numerical simulation of the 3D electrostatic potential of bio-molecules.
The PBE applies to 
a solvated bio-molecular system modeled by dielectrically separated  domains
with singular Coulomb potentials distributed in the molecular region.
We consider a system embedded into a rectangular domain 
$\Omega = \overline{\Omega}_m \cup \overline{\Omega}_s$
with boundary $\partial \Omega$, see Fig. \ref{fig:Protein}, where the molecule (solute) region 
$\Omega_m$ is a union of small spheres of radius $\sigma >0$ (Van der Waals sphere of a solute atom) 
surrounding each individual atom in the system. The solute region is immersed into the 
the solvent region $\Omega_s$.

The linearized Poisson-Boltzmann equation for the electrostatic potential of a bio-molecule, 
$u(x)$, $x\in \mathbb{R}^3$, takes a form, see \cite{Holst:94},
\begin{equation}\label{eqn:PBE}
 -\nabla\cdot(\epsilon \nabla u)+ \kappa^2 u=\rho_f= 
 \sum\limits_{{k}=1}^N q_k \delta(\|{x} - {x}_{k}\|)
 \quad \mbox{in } \quad \Omega \in \mathbb{R}^3,
\end{equation}
where $\rho_f $
% $$
% \rho_f= \sum\limits_{{k}=1}^N z_k \delta(\|{x} - {x}_{k}\|), \quad z_k \in \mathbb{R}
% $$ 
is the scaled singular charge distribution supported at points $x_k$ in $\Omega_m$, 
$\delta$ is the Dirac delta, and $q_k\in \mathbb{R}$ denotes the charge located at 
the atomic center $x_k$. 
Here $\epsilon_m=O(1)>0$ and $\kappa=0$ in $\Omega_m$, while in the solvent region $\Omega_s$ 
we have $\kappa \geq 0$ and $\epsilon_s \geq \epsilon_m$ 
(the typical ration $\epsilon_s/\epsilon_m$ could be about several tens, say,
about $80$ for water).
Here $\kappa$ is the so-called Debeye length that characterizes the screening effect 
(due to presence of the ions).
The interface conditions on the interior boundary 
$\Gamma=\partial \Omega_m$ arise from the dielectric theory (continuity of potentials and flaxes):
\begin{equation}\label{eqn:Intcond_PBE}
 [u]=0, \quad \left[ \epsilon\frac{\partial u}{\partial n}  \right]=0 \quad \mbox{on}\quad \Gamma.
\end{equation}
The boundary conditions on the external boundary
 are specified depending on the particular problem setting.
The simplest choice is either the Dirichlet  $u_{| \partial \Omega}=0$ or mixed 
$\frac{\partial u}{\partial n} = \mu u$ boundary conditions on $\partial\Omega$.

% We shall also consider the practically interesting inhomogeneous Dirichlet boundary conditions 
% taking a form
% \begin{equation}\label{eqn:Dir_bc_PBE}
%  u(x)_{| \partial \Omega}= \frac{1}{4 \pi \epsilon_m} \sum\limits_{{k}=1}^N 
%  \frac{z_k e^{-\kappa \|{x} - {x}_{k}\|}}{\|{x} - {x}_{k}\|}, \quad x \in \partial \Omega.
% \end{equation}

The traditional numerical approaches to PBE are based on either multigrid  
\cite{Holst:94,Lu2008} or domain decomposition \cite{CaMaSt:13,LiStCaMaMe:13} methods,
combined with certain 
regularization schemes diminishing the strong singularity in the right-hand side.

% We shall also consider the practically interesting inhomogeneous Dirichlet boundary conditions 
% taking a form
% \begin{equation}\label{eqn:Dir_bc_PBE}
%  u(x)_{| \partial \Omega}= \frac{1}{4 \pi \epsilon_m} \sum\limits_{{k}=1}^N 
%  \frac{z_k e^{-\kappa \|{x} - {x}_{k}\|}}{\|{x} - {x}_{k}\|}, \quad x \in \partial \Omega.
% \end{equation}

In what follows, we briefly sketch the new regularization scheme based on the RS splitting 
of the discretized Dirac delta described in \S\ref{sec:RS_Dirac_discr}. 
To avoid the unessential technical details, we present the main construction on the 
continuous level
% in the case of single charge, i.e. $\rho_f=  z_0 \delta(\|{x}\|)$, 
such that our regularization scheme applies in the same way to the FEM discretization 
of the PBE.

Let $u_f=u_{\mbox{\footnotesize short}} + u_{\mbox{\footnotesize long}}$ 
be the RS splitting of the free space potential generated 
by the density $\frac{1}{\epsilon_m}\rho_f$, i.e., satisfying the equation 
$$
-\epsilon_m \Delta u_f = \rho_f\quad   \mbox{in} \quad  \mathbb{R}^3.
$$ 
This introduces the 
corresponding RS decomposition of the density $\rho_f$ in (\ref{eqn:PBE}) due to the equation
\[
 -\epsilon_m \Delta (u_{\mbox{\footnotesize short}} + u_{\mbox{\footnotesize long}}) = \rho_f =
 \rho_{\mbox{\footnotesize short}} + \rho_{\mbox{\footnotesize long}},
\]
where
\[
 \rho_{\mbox{\footnotesize short}} = -\epsilon_m \Delta u_{\mbox{\footnotesize short}}, \quad
 \mbox{and} \quad 
 \rho_{\mbox{\footnotesize long}}=-\epsilon_m \Delta u_{\mbox{\footnotesize long}}.
\]
Now we introduce the new regularization scheme by the RS splitting of the solution $u$ in the form
\[
 u= u_{\mbox{\footnotesize short}} + \overline{u}_{\mbox{\footnotesize long}},
\]
where the component $\overline{u}_{\mbox{\footnotesize short}}$ is precomputed explicitly 
and stored in the canonical/Tucker tensor format living on the fine Cartesian grid,  
and the unknown function $\overline{u}_{\mbox{\footnotesize long}}$ satisfies the equation 
with the modified right-hand side 
\begin{equation}\label{eqn:PBE_regul}
 -\nabla\cdot \epsilon \nabla \overline{u}_{\mbox{\footnotesize long}} + 
 \kappa^2 \overline{u}_{\mbox{\footnotesize long}}=\rho_{\mbox{\footnotesize long}} 
 \quad \mbox{in } \quad \Omega,
\end{equation}
and equipped with the same interface and boundary conditions as the initial PBE.

% Furthermore, the modified density 
% $\rho_{\mbox{\footnotesize long}} = -\epsilon_m \Delta u_{\mbox{\footnotesize long}}$ 
% in the right-hand side in (\ref{eqn:PBE_regul}) is the sum of
% $N$ functions each localized within the corresponding ``Van der Waals region'' centered at $x_k$.
% Hence the effective support of $\rho_{\mbox{\footnotesize long}}$ is located withing
% the molecular region $\Omega_m$.

We conclude with the following 
\begin{lemma}\label{lem:PBE_RSTensor}
Let the effective support of the short-range component ${\bf P}_{R_s}$ in the reference
 potential ${\bf P}_R$ be chosen  not larger than $\sigma >0$ (the radius of Van der Waals spheres), 
 then: 
 
 (A) the interface conditions in the regularized formulation of the PBE in (\ref{eqn:PBE_regul})
 remain the same as in the initial formulation.

 (B) The effective support of the modified right-hand side in (\ref{eqn:PBE_regul}) is included into $\Omega_m$.
\end{lemma}
\begin{proof}
(A). Since $u_{\mbox{\footnotesize short}}$ is localized within each sphere $S_k$ 
of radius $\sigma$ centers at $x_k$ (i.e., $u_{\mbox{\footnotesize short}}$ 
and its co-normal derivative both vanish on $\Gamma$), we deduce that 
$\overline{u}_{\mbox{\footnotesize long}} = u - u_{\mbox{\footnotesize short}}$ inherits the same
interface conditions on $\Gamma$ as the solution $u$ of PBE (see also arguments in the proof of 
Lemma \ref{lem:BounCond_PoissonMod}).

Furthermore, the modified density 
$\rho_{\mbox{\footnotesize long}} = -\epsilon_m \Delta u_{\mbox{\footnotesize long}}$ 
in the right-hand side in (\ref{eqn:PBE_regul}) is the sum of
$N$ functions each localized within the corresponding ``Van der Waals region'' centered at $x_k$.
Hence the effective support of $\rho_{\mbox{\footnotesize long}}$ is located withing
the molecular region $\Omega_m$, which proves item (B).
\end{proof}

For justification of the important property (B), 
we provide the numerical tests which indicate the decrease of the effective support 
of the long-range part in the discretized Dirac delta if the corresponding mesh size tends to zero,
see Figure \ref{fig:Long_Delta_grids}.

The discrete version of the regularization scheme introduced above requires the following calculations:

(A) Computation and storage of $\overline{u}_{\mbox{\footnotesize short}}$  
in the RS canonical tensor format living on the fine Cartesian grid.

(B) Computation and storage on the fine Cartesian grid 
of the modified right-hand side $\rho_{\mbox{\footnotesize long}}$ by application 
the Discrete Laplacian to the long-range part of the free space potential $u_{\mbox{\footnotesize long}}$.

(C) FEM discretization of the regularized equation (\ref{eqn:PBE_regul}) on a coarse enough grid, where the 
 \emph{smooth right-hand side} is obtained by the simple interpolation of $\rho_{\mbox{\footnotesize long}}$
from the fine Cartesian grid 
(where the tensor representation of $\rho_{\mbox{\footnotesize long}}$ was precomputed) 
onto the coarse FEM mesh.

The advantage of the presented regularized formulation is threefold: 
\begin{itemize}
 \item The absence of singularities in the solution $\overline{u}_{\mbox{\footnotesize long}}$,
 i.e., the relatively coarse grid provides the accurate FEM solution.
 \item The localization of the solution splitting into the domain $\Omega_m$ only, i.e.,
 the interface and boundary conditions remain unchanged.
 \item The most singular part in the solution is precomputed with high accuracy and stored 
 in the RS tensor format.
\end{itemize}

In \cite{BeKhKhKwSt:18}, we test numerically
the FEM based computational scheme  that allows to implement 
the presented regularization strategy in the easy and effective way by using the 
multigrid solver.

\section{RS tensor decomposition of the elliptic operator inverse}
\label{sec:RS_Resolvent}

The equation (\ref{eqn:OD_Dirac}) implies the integral representation of 
the elliptic operator inverse  (elliptic resolvent) 
${\cal L}^{-1}$ in $\mathbb{R}^d$, such that the solution $u(x)$ of the equation
\begin{equation}\label{eqn:L-BVP}
 {\cal L} u = f(x), \quad x\in \mathbb{R}^d, \quad |u(x)|\to 0 \quad \mbox{as} \quad \|x\|\to \infty
\end{equation}
allows the integral convolution type representation 
\begin{equation}\label{eqn:L-potential}
 u (x)= ({\cal L}^{-1} f)(x) = \int_{\mathbb{R}^d}p(\|x-y\|)f(y) dy, % \quad {\cal R}_d= {\cal L}^{-1}
\end{equation}
for integrable, compactly supported  density $f$ in the right-hand side, where $p(\|x\|)$ is the the fundamental solution
of the elliptic operator ${\cal L}$, satisfying the equation (\ref{eqn:OD_Dirac}). 
Assume that we are given the RS type splitting of the fundamental solution $p(\|x\|)=p_s(\|x\|) + p_l(\|x\|)$
in (\ref{eqn:RS_Green}), then we are able to introduce the RS decomposition of the elliptic operator inverse
\[
 {\cal L}^{-1} = {\cal L}^{-1}_s + {\cal L}^{-1}_l,
\]
where the local  ${\cal L}^{-1}_s $ and the long-range ${\cal L}^{-1}_l$ components of 
${\cal L}^{-1}$ are defined by using (\ref{eqn:L-potential}),
\begin{equation}\label{eqn:RS_resolvent}
 u (x)= ({\cal L}^{-1} f)(x) = 
 \int_{\mathbb{R}^d}(p_s(\|x-y\|) + p_l(\|x-y\|)) f(y) dy \equiv u_s +u_l, % {\cal L}^{-1}_s  f + {\cal L}^{-1}_l  f,
\end{equation}
such that
\[
  u_s = \int_{\mathbb{R}^d} p_s(\|x-y\|) f(y) =: {\cal L}^{-1}_s  f, \quad 
  u_l = \int_{\mathbb{R}^d} p_l(\|x-y\|) f(y) =: {\cal L}^{-1}_l  f.
\]
This local-global splitting of the elliptic solution operator essentially relies on 
the RS decomposition of the Green kernel, which can be derived for the wide class of elliptic operators,
as demonstrated in the following sections. 

The new regularization scheme for the elliptic operator inverse 
allows to simplify the solution methods for PDEs with piecewise constant 
coefficients and highly non-regular right-hand sides by introducing 
the solution decomposition (regularization) $u= u_s + \widehat{u}$, where $u_s$ 
can be easily precomputed and the updated right-hand side in the equation for the 
new unknown function $\widehat{u}$ remains smooth.
We explain this regularization scheme in the more detail on the example in \S\ref{ssec:Applic1_Poisson}.

\subsection{The  Poisson problem with localized right-hand side}\label{ssec:Applic1_Poisson}

As a simple example, we show how the RS tensor decomposition of the discretized Newton kernel 
by (\ref{eqn:Split_Tens}) can be directly applied to the solution of Poisson problem 
\begin{equation} \label{eqn:Poisson}
 -\Delta \widehat{u } =f,\quad \mbox{in} \quad \Omega \in \mathbb{R}^3,
\end{equation}
where $f\in L_{loc}^2(\mathbb{R}^3)$ is the non-regular function with \emph{compact support} 
restricted into 
a subdomain of complicated geometry embedded in  $\Omega$ with $\Gamma =\partial \Omega$. In particular,
$f$ may by defined on lower dimensional manifold.

For the ease of exposition, in what follows we first describe the idea on the continuous level.
Notice that every RS splitting of the Newton kernel
$1/\|x\|= p_{s}(\|x\|) + p_{l}(\|x\|)$ generates the corresponding decomposition of the free space
solution $u(x)$ in $\mathbb{R}^3$, satisfying (\ref{eqn:L-BVP}) with ${\cal L}=-\Delta$,
\[
 {u }(x)= \int_{\mathbb{R}^3}(p_{s}(\|x-y\|) + p_{l}(\|x-y\|))f(y) dy = {u}_s + {u}_l. 
\]
Here the short-range part
\[
{u}_s= \int_{\mathbb{R}^3} p_{s}(\|x-y\|)f(y) dy = (-\Delta)_s^{-1} f
\]
can be calculated on a grid by direct convolution transform with strongly localized 
convolving kernel $p_{s}(\|x\|)$ and, hence, has almost the same support 
as the right-hand side $f$ (i.e., it remains well localized), 
while the function ${u}_l$ has the improved regularity. 

Now, we represent the solution of (\ref{eqn:Poisson}) in the form 
$$
\widehat{u }={u}_s + \bar{u },
$$ 
where $\bar{u }$ satisfies the Poisson equation with the modified right-hand side
$\bar{f}= f + \Delta {u}_s$,
\begin{equation} \label{eqn:Poisson_reg}
 -\Delta \bar{u } = \bar{f} = f + \Delta {u}_s = (I - {\cal L}{\cal L}^{-1}_s)f
 \quad \mbox{in} \quad \Omega \in \mathbb{R}^3,
\end{equation}
and with the same boundary conditions on $\Gamma$ as the initial equation (\ref{eqn:Poisson}). 
\begin{lemma}\label{lem:BounCond_PoissonMod}
Suppose that 
$$
\mbox{dist}(\mbox{supp}f,\Gamma)> \sigma,
$$ 
where $\sigma>0$ denotes the effective support 
of the short-range part ${\bf P}_{R_s}$ in the tensor representation of the Newton kernel, 
 then:\\
(A) The trace of the 
 regularized solution $\bar{u }$ and its normal derivatives on $\Gamma$ 
 coincide with those for the solution $\widehat{u}$ of the equation (\ref{eqn:Poisson}).\\
 (B) The support of the modified right-hand side $\bar{f}$ is strongly embedded into $\Omega$.\\
 (C) The function $\bar{u }$ has as many partial derivatives in the $\sigma$-vicinity
 of the  domain $\mbox{supp}(f)$,
 as the the long-range part of the Green kernel, $p_{l}(\|x\|)$.
\end{lemma}
\begin{proof}
Indeed, the statement (A) is a consequence of the fact that $u_s=0$ on $\Gamma$ and, moreover,
$u_s$ vanishes in some vicinity of $\Gamma$  since $\mbox{supp} u_s$ is strongly embedded in $\Omega$.
Hence, the boundary conditions in the regularized problem (\ref{eqn:Poisson_reg}) remain unchanged.

Item (B) follows from the localization property 
of the function $\Delta {u}_s$, which, in turn, is the consequence of the localization 
of the (discretized) convolving kernel $p_{s}(\|x\|)$, that 
is $\mbox{diam}(\mbox{supp}\,p_{s})\leq \sigma$, and the relation
\[
 \mbox{dist}(\mbox{supp}\bar{f},\Gamma)\geq \mbox{dist}(\mbox{supp}f,\Gamma) - \sigma >0.
\]

To prove (C), we notice that  
\[
 \frac{\partial^k}{\partial x_k} u_l (x)= 
 \int_{\mathbb{R}^3}  \frac{\partial^k}{\partial x_k} p_{l}(\|x-y\|)\, f(y) dy,
\]
and, hence, $u_l$ has as many partial derivatives as the long-range part of the kernel, 
$p_{l}(\|x\|)$, which has, in general, high regularity or even it is globally analytic. 
This feature can be then reconstructed to the function $\bar{u}$ by standards arguments based on the properties 
of the harmonic extension.
\end{proof}

This scheme also applies to the discrete setting by using the RS tensor splitting of 
the discrete Newton kernel, ${\bf P}_R= {\bf P}_{R_l} + {\bf P}_{R_s}$, represented on fine enough 
$n \times n \times n$ tensor grid including the computational domain $\Omega$. 
To that end, we consider the discretized representation of (\ref{eqn:L-potential}), 
which then leads to the discrete version of 
the above regularization scheme by representing the FEM-Galerkin solution 
$\widehat{\bf u}$ of (\ref{eqn:Poisson}) in the form
\begin{equation} \label{eqn:Poisson_reg_discr}
 \widehat{\bf u} = {\bf P}_{R_s} \ast {\bf f} + \bar{\bf u}.
\end{equation}
Here the vector $\bar{\bf u}$ satisfies the discrete Poisson equation in 
$\Omega$ with the modified right-hand side $\bar{\bf f}= {\bf f} + \Delta_h {\bf P}_{R_s}$, 
\begin{equation} \label{eqn:Poisson_reg_FEM}
 - \Delta_h \bar{\bf u} = \bar{\bf f},
\end{equation}
and the same boundary conditions on $\Gamma$ as in the initial equation (\ref{eqn:Poisson}).
The regularized scheme (\ref{eqn:Poisson_reg_FEM}) can be applied to the equations with
piecewise constant (or even nonlinear) coefficients if $\mbox{supp}( f)$ is embedded onto
the subdomain with constant coefficient.
The solution process for the equation (\ref{eqn:Poisson_reg_FEM}) 
can be implemented in rank structured tensor formats when appropriate.

Finally, we summarize the beneficial properties of the regularization scheme (\ref{eqn:Poisson_reg_discr}):
\begin{itemize}
 \item The right-hand side in the regularized equation has much better regularity compared 
 with the possibly non-regular function $f$.
 \item The boundary conditions in the regularized equation remain the same as in the initial problem.
 \item The complexity of the calculation the modified right-hand side (3D convolution with highly
 localized convolving kernel) is proportional to the number of grid points in the active support of 
 the initial right-hand side $f$.
 \item Rank structured tensor formats can be applied in the solution process 
 if the geometry of the domain $\Omega$ is appropriate.
\end{itemize}

In the following sections we recall the examples of classical Green kernels which can be represented 
in the low-rank RS tensor format.

% Finally, we notice that the \emph{range separation principle} can be generalized to 
% more than two-term splitting, taking into account the  requirements of specific applications.
%  

\subsection{Equations with the advection-diffusion operators  }\label{ssec:PBE_Applic2}

Consider the fundamental solution of the advection-diffusion operator ${\cal L}_d$ with constant coefficients 
in $\mathbb{R}^d$
\[
 {\cal L}_d = -\Delta + 2\bar{b} \cdot \nabla + \kappa^2, \quad \bar{b}\in \mathbb{C}^d, \quad 
 \kappa\in \mathbb{C}.
\]
If $\lambda^2= \kappa^2 + |\bar{b}|^2 \neq 0$, the the fundamental solution (Green kernel) 
$\kappa_\lambda(x)$ of the operator ${\cal L}_d$ takes a form
\[
 \eta_\lambda(x) = 
 \frac{e^{\langle \bar{ b},{x} \rangle}}{(2\pi)^{d/2}}\left( \frac{\|{x}\|}{\lambda} \right)^{1-d/2}
 K_{d/2-1}(\lambda \|{ x}\|),\quad x\in \mathbb{R}^d,
\]
where $K_\nu$ is the modified Bessel function of the second kind (the Macdonald function) \cite{AbrSte}.
Here, we use the conventional notation 
$\langle { y}, { z} \rangle = \sum_{\ell=1}^d y_\ell z_\ell$, $\|{ z}\|^2= 
\langle { z}, { z} \rangle$.

If $\kappa^2 + |\bar{b}|^2 = 0$, then for $d\geq 3$ it holds
\[
 \eta_0(x) = \frac{1}{(d-2) \omega_d} \frac{e^{\langle \bar{b},x \rangle }}{\|x\|^{d-2}},
\]
where $\omega_d$ is the surface area of the unit sphere in $\mathbb{R}^d$.
Notice that the radial function $\frac{1}{\|x\|^{d-2}}$ for $d\geq 3$ allows the RS decomposition of the 
corresponding discrete tensor representation based on the sinc quadrature approximation, which implies
the RS representation of the kernel function $\eta_0(x)$, 
since the function $e^{\langle \bar{b},x \rangle }$ is already separable.

In the particular case $\bar{b}=0$, we obtain the fundamental solution of the operator
${\cal L}_3 = -\Delta + \kappa^2$ for $d=3$, also known as the Yukawa 
(for $\kappa\in \mathbb{R}$) or Helmholtz (for $\kappa\in \mathbb{C}$) Green kernels
\[
\eta_\lambda(x) =  \frac{1}{4\pi} e^{-\kappa \|x\|}/\|x\|, \quad x\in \mathbb{R}^3.
\]
 In the case of Yukawa  kernel the tensor representations
by using Gaussian sums  have been discussed in \cite{KhorBook:17,VeBokh_book:18} 
and implemented in \cite{BeKh:08}.

The  Helmholtz equation with $\mbox{Im}\, \kappa >0$ (corresponds to the diffraction potentials) 
arises in problems of acoustics, electromagnetics and optics.
We refer to \cite{MazSch:book} for the detailed discussion of this class of fundamental solutions. 
Fast algorithms for the oscillating Helmholtz kernel have been considered in  \cite{KhorBook:17} and  
\cite{BelkKuMon:08}. However, in this case the construction of the RS tensor decomposition
remains an open question.

\subsection{Elastic and hydrodynamics potentials in $\mathbb{R}^3$  }\label{ssec:PBE_Applic3}

In elasticity theory, linear isotropic and homogeneous elastic problems in whole space $\mathbb{R}^3$ 
are governed by the Lam{\'e} system of equations
\[
 \mu \Delta {u} +(\lambda + \mu) \mbox{grad} \, \mbox{div} {u} = { f},
\]
where ${u}$ is the displacement vector, ${ f}$ is the volume force, $\lambda$ and $\mu$ are the 
Lam{\'e} constants. The solution of the Lam{\'e} system is given by the volume potentials
in the form of convolution integrals
\[
 u_k(x) = \int_{\mathbb{R}^3} \Gamma_{k\ell} (x-{ y}) f_\ell({ y}) d { y}, 
\]
where $\| \Gamma_{k\ell}\|_{3\times 3}$ is the Kelvin-Somigliana fundamental matrix with
\[
 \Gamma_{k\ell} (x) = \frac{\lambda + \mu}{8\pi \mu (\lambda +2\mu)}
 \left(\frac{\lambda +3\mu}{\lambda +\mu} \frac{\delta_{k\ell}}{\|x\|} + 
 \frac{x_k x_\ell}{\|x\|^3}\right).
\]
We refer to the monographs \cite{MazSch:book,HsiWendBOOK:08,SauterSchwa_book:11}, where the detailed 
analysis of these fundamental solutions have been presented. 
\begin{remark} \label{rem:RS_Lame}
For our discussion it is worth to notice that
the radial functions $\frac{1}{\|x\|}$ and $\frac{1}{\|x\|^3}$, included in the above 
representation, allow the low-rank RS tensor decompositions, which can be used for the solution of corresponding 
equations in non-homogeneous  media and with the right-hand side composed of many pointwise 
singular densities (forces).
\end{remark}

In the  case of 3D  biharmonic operator ${\cal L}= \Delta^2$ the fundamental solution reads as
\[
 p(\|x\|)=- \frac{1}{8\pi} \|x\|, \quad x\in \mathbb{R}^3.
\]

The hydrodynamic potentials correspond to the linearized classical Navier-Stokes equations
\[
 \nu \Delta {u} -\mbox{grad} \, p = { f}, \quad \mbox{div}\, {u} =0,
\]
where ${u}$ is the velocity field, $p$ denotes the pressure, and $\nu$ is the constant viscosity coefficient.
The solution of the Stokes problem in $\mathbb{R}^3$ can be expressed by the hydrodynamic potentials
\begin{equation}
 u_k (x) = \int\limits_{\mathbb{R}^3} \sum\limits^3_{\ell=1}
 \Psi_{k\ell} (x -{y}) f_\ell ({y}) \mbox{d}{y},\quad 
 p(x) = \int\limits_{\mathbb{R}^3} 
 \langle \Theta (x -{y}),{ f} ) ({y})\rangle \mbox{d}{y}
 \end{equation}
 with the fundamental solution
 \begin{equation}
  \Psi_{k\ell} (x) =\frac{1}{8\pi \nu} 
  \left(\frac{\delta_{k\ell}}{\|x\|} + \frac{x_k x_\ell}{\|x\|^3}\right), \quad
  \Theta (x)= \frac{x}{4\pi \|x\|^3}, \quad x\in \mathbb{R}^3.
 \end{equation}
 
The existence of the low-rank RS tensor representation for both the biharmonic and 
the hydrodynamic potential is based on the same argument as in Remark \ref{rem:RS_Lame}.

\section{Conclusions}\label{sec:Conclusions}

The  method for the operator dependent RS tensor approximation of the Dirac delta in $\mathbb{R}^d$
introduced in this paper
allows to represent the discretized $\delta$-function as a sum of the 
short- and long-range rank-structured tensors.
Each of these two tensors reproduces the corresponding parts of  the RS  tensor  representation for 
the discretized Green kernel associated with the target elliptic operator.

We outline how this method can be applied 
to solving the potential equations in non-homogeneous media 
when the density in the right-hand side is given by the large sum of pointwise singular charges.
The approach can be utilized in various applications modeled by the elliptic 
problems with jumping (or even nonlinear) coefficients and non-regular right-hand side 
as described in Section \ref{sec:RS_Resolvent}, where the RS splitting of the elliptic operator 
inverse is introduced.  

We show that the tensor approximation of the Dirac delta can be efficiently
computed by applying the discretized Laplace operator to
the canonical RS tensor representation of the Newton kernel on
$n\times n \times n$ 3D Cartesian grid.  
Due to the RS tensor representation, complexity of this numerical 
approximation scales as $O(n)$ in a grid size.
  
Since the rank of the long-range  part of the RS tensor representation representing 
the collective electrostatic potentials of 
a multi-particle system depends only logarithmically on the number of particles,
the same property  is inherited by the long-range part of the ``collective Dirac delta''.
It means that the corresponding computational work for numerical treatment of a 
multiparticle system increases only linearly-logarithmically in the number of particles.  

In this way the approach can be applied to the Poisson-Boltzmann equation 
modeling the electrostatic potential of a bio-molecule.
Other possible applications are based on the RS tensor representation of
elastic and hydrodynamics potentials in elasticity problems in $\mathbb{R}^3$,
as well as of the fundamental solution of advection-diffusion operator with 
piecewise constant coefficients in $\mathbb{R}^3$.
The presented regularization scheme also applies to the equations with the continuous density
in the right-hand that might have very low regularity and the complicated shape of the effective 
support (say, electrostatics and  magnetostatics).

\vspace{0.3cm}
   
{\bf Acknowledgements.}
The author would like to thank Dr. Venera Khoromskaia
(MPI MiS, Leipzig) for useful discussions and numerical  simulations.

\begin{footnotesize}

\end{footnotesize}
 
\end{document}